\documentclass{article}

\usepackage{amsmath}
\usepackage{amssymb}
\usepackage{amsthm}
\usepackage{amsfonts}
\usepackage{mathtools}
\usepackage[english]{babel}
\usepackage{enumitem}
\usepackage{anysize}
\usepackage{natbib}
\usepackage[
colorlinks=true,
linkcolor=blue,
citecolor=blue,
urlcolor=blue
]{hyperref}

\theoremstyle{definition}
\newtheorem{definition}{Definition}[section]
\newtheorem{theorem}[definition]{Theorem}
\newtheorem{lemma}[definition]{Lemma}
\newtheorem{proposition}[definition]{Proposition}

\newtheorem{remark}[definition]{Remark}

\newtheorem{example}[definition]{Example}

\DeclareMathOperator*{\limess}{ess\,lim}
\DeclareMathOperator*{\blim}{\operatorname{B-lim}}

\newcommand{\R}{\mathbb{R}}

\newcommand{\C}[1]{C([0,1];\R^{#1})}
\newcommand{\AC}[1]{W^{1,1}([0,1];\R^{#1})}
\newcommand{\BV}[1]{\operatorname{BV}([0,1];\R^{#1})}

\newcommand{\Linf}[1][]{L^{\infty}([0,1];\R^{#1})}
\newcommand{\Lone}[1][]{L^{1}([0,1];\R^{#1})}
\newcommand{\leb}{m}

\newcommand{\ba}{\operatorname{ba}([0,1],\mathfrak{L},\leb)}
\newcommand{\mc}[1]{\mathcal{#1}}

\date{}

\title{Nondegeneracy Conditions for Control Problems with
	Nonregular Mixed Constraints }
\author{J.A. Becerril$^1$\quad K.L. Cortez$^2$\\
	\small{$^1$Instituto Tecnol\'ogico y de Estudios Superiores de Monterrey, Departamento de Ciencias e Ingenier\'ia,}\\ \small{Estado de M\'exico, M\'exico}, jorge.becerril@tec.mx\\
	\small{$^2$Universidad Nacional Aut\'onoma de México, Departamento de Física-Matem\'atica-IIMAS, Ciudad de M\'exico}\\
}

\begin{document}
	\maketitle
	\begin{abstract}
        We establish nondegeneracy and normality conditions for optimal control problems with a single nonregular mixed inequality constraint, where purely finitely additive set functions (charges) appear as multipliers. We first prove for a model problem nonregular at a single instant that degeneration is intrinsic and unavoidable: every admissible process satisfies the necessary conditions degenerately via a pure charge, and normalized multipliers of natural regular approximations admit no weak$^*$ limit in $L^\infty([0,1];\mathbb{R})^*$, yielding only pure charges under every generalized limit, in the sense of Banach limits. Motivated by this analysis, we propose four verifiable conditions that progressively guarantee nondegenerate multipliers, the vanishing of the pure charge, normality, and explicit bounds on charge mass. Furthermore, we show that at a nonregular terminal instant, a balance identity ties the charge mass to the cost multiplier, deciding between complete abnormality and normality. We provide examples confirming the realizability and discussing the limits of each condition.
	\end{abstract}
	
	
	
	\section{Introduction}
	
	In optimal control problems with pure state constraints, the phenomenon of degeneration has been extensively studied due to its significant implications for deriving meaningful first-order necessary optimality conditions. Notable contributions in this area include \cite{a6,ferreira,fontes2013,fontes2015,R-V}. Degeneration typically arises when the initial point of the trajectory is restricted to lie on the boundary of the admissible region. In such cases, for any admissible process, it is possible to find nonzero yet degenerate multipliers that satisfy the Pontryagin Maximum Principle (PMP). This renders the necessary optimality conditions uninformative for identifying potential minimizers. To address this issue, numerous results in the literature provide conditions to ensure nondegeneracy, i.e., to guarantee that for a minimizer, it is possible to find nondegenerate multipliers that satisfy the necessary conditions, thereby ensuring their applicability. Other important studies in this context include \cite{lopes2011,R-V2}.
	
	When it comes to problems with mixed state/control constraints, most of the literature has traditionally imposed some kind of regularity conditions to derive necessary optimality conditions; see, for example, \cite{a5, b1, c1, dm2, l1}. Regularity assumptions are helpful in simplifying the analysis, but often at the expense of broad applicability. As explained in \cite{becerriletal2022}, the specific form of regularity depends on the problem context and can be framed as a constraint qualification or as a condition that applies to a given nominal process. For instance, common examples include the Mangasarian--Fromovitz constraint qualification and the regularity condition proposed in \cite{dm2} for cases involving smooth mixed constraints. In settings where non-smooth constraints are present, relevant assumptions are the bounded slope condition and the weak basic constraint qualification as defined in \cite{c1} and \cite{l1}, respectively. In the particular context of a fully convex problem, a Slater-type condition was used in \cite{behe2022} to derive first order optimality conditions, which proved to be less restrictive in some ways; for instance, it allows for costates of bounded variation, in contrast to other constraint qualifications (including the Mangasarian--Fromovitz constraint qualification and the bounded slope condition) that renders it absolutely continuous. 
	
	Recently, there has been a surge in research focusing on nonregular mixed constraints, expanding the applicability of optimality conditions to more general settings. Examples of this trend can be found in \cite{be2021, dmitruk2022, becerriletal2022, dmitruk2023, dmitry2023}. In \cite{be2021}, first-order optimality conditions were established in the form of a weak minimum principle without imposing regularity conditions. Instead, a closedness condition on the image of a functional associated with the equality constraints was sufficient. The term ``weak'' here indicates the absence of the Weierstrass condition (also referred to as the minimality condition) in the principle. Building on this, \cite{becerriletal2022} extended these results by deriving a minimum principle that includes the Weierstrass condition, under the assumptions of affine Lagrangian, dynamics, and constraints on the control variable and a full-rank condition on the equality constraints. Meanwhile, \cite{dmitry2023} addressed the case where a non-regular linear inequality mixed constraint is considered, establishing necessary conditions and proving that the Weierstrass condition holds on specific subsets. In \cite{dmitruk2022} and \cite{dmitruk2023}, a weak minimum principle was also formulated. However, in these works, measures appear as multipliers associated with the mixed constraints, contrasting with \cite{be2021}, where pure charges emerged. A different line, closer to numerical practice, is pursued in \cite{moreira2024}, where necessary conditions of \emph{asymptotic} type are provided in the form of a weak maximum principle which does not presuppose a constraint qualification, together with a method of multipliers realizing them numerically. This approach is in the spirit of the sequential optimality conditions of mathematical programming.
	
	However, despite their generality, these recent approaches are not without limitations. For instance, in \cite[Theorem 4.1]{be2021}, it has been noted that, under certain conditions unrelated to optimality, it is possible to identify degenerate multipliers that satisfy the weak minimum principle for some admissible processes (see Section \ref{sec:when}). This phenomenon, which we refer to as degeneration in the context of mixed constraints, bears strong similarities to the well-known degeneration observed in pure state constraints. In both cases, the existence of degenerate multipliers undermines the utility of the necessary conditions in effectively capturing optimality. For this reason, we adopt the same terminology to emphasize the parallel between these two scenarios.
	
	To the best of our knowledge, the question of how to preclude this phenomenon has not been addressed in the nonregular \emph{mixed-constrained} setting. For pure state constraints the corresponding theory is by now well developed and rests on inward-pointing constraint qualifications \cite{ferreira,fontes2013,fontes2015,nat-colombo-rampazzo,lopes2011,ledzewicz1993}; for mixed constraints, degeneracy and normality have been studied under regularity hypotheses or rank-type constraint qualifications, that is, precisely under the assumptions that \cite{be2021,dmitruk2022,becerriletal2022,dmitruk2023,dmitry2023} dispense with. The closest antecedent we are aware of is \cite{andreani2024}, where nondegenerate necessary optimality conditions are obtained for \emph{discrete-time} problems with mixed constraints under the constant rank of the subspace component constraint qualification. The conditions proposed below appear to be the first to guarantee nondegeneracy, or normality, in the presence of nonregular mixed constraints. Two features distinguish this regime from the state-constrained one. First, the multipliers live in $\Linf^*$, whose closed unit ball is not weak$^*$ \emph{sequentially} compact, so that the passage to the limit in the multipliers is itself delicate and call for either weak$^*$ subnets or Banach limits. Since Banach limits are closely related to charges, we favor them in this paper. Second, and more importantly, the onset of degeneration is governed by an \emph{integrability threshold} on $|\bar g_u|^{-1}$ near the nonregular instant; a mechanism with no counterpart in the pure state-constrained.

	The contributions of this paper are threefold. First, we analyze in depth an example taken from \cite{be2021} which is nonregular at a single instant, and we show that it is completely degenerate: every admissible process satisfies the necessary conditions of \cite[Theorem 4.1]{be2021} with multipliers whose only nonzero component is a pure charge concentrated at the nonregular instant. We further prove that this degeneration is intrinsic: when the problem is approximated by regular problems whose multipliers are unique up to scaling and normal, the corresponding normalized multipliers admit no weak$^*$ limit in $\Linf^*$, and every generalized limit, in the sense of Banach limits, is precisely a degenerate multiplier. This analysis identifies the mechanism behind the degeneration. Second, guided by the previous findings, we propose conditions with various degrees of restrictiveness that guarantee nondegeneracy, normality and the vanishing of the charge for a single mixed constraint. These conditions draw on established results for nondegeneracy in problems with pure state constraints (see \cite{ferreira, ledzewicz1993}), allowing us to extend these insights to the case of mixed constraints. Third, we exhibit an asymmetry between the two endpoints of the interval. At a nonregular \emph{terminal} instant the nonintegrability of $|\bar g_u|^{-1}$ does not produce complete degeneration; instead, a \emph{terminal balance identity} (Proposition \ref{prop:termbalance}) ties the mass of the charge multiplier around $t=1$ to the cost multiplier, and decides between complete abnormality and normality with a charge of prescribed mass. Theorem \ref{t5} connects the two sides; it covers processes that are nonregular at both endpoints simultaneously and secures a normal tuple whose charge is confined to the terminal instant, together with a bound on its mass. We illustrate the results with families of examples that show that every hypothesis is realizable and clearly delimit the reach of each condition.
	
	The paper is structured as follows. Section \ref{prelim} collects notation, the elements of the theory of finitely additive measures used throughout, Banach limits, the Yosida--Hewitt decomposition, the statement of the problem, and the nonregular weak maximum principle of \cite{be2021} on which our analysis rests. Section \ref{sec:when} describes the degeneracy phenomenon: it locates the degenerate tuples (Proposition \ref{prop:trivial}), dissects the case study (Subsection \ref{subsec:example}), and establishes the terminal balance identity that governs abnormality at a nonregular terminal instant. Section \ref{sec:nondeg} presents the nondegeneracy and normality conditions: first at a nonregular initial instant, by approximation with regular truncations; then via a linearized margin condition that dispenses with approximation altogether and applies to every tuple of multipliers; and finally in the mixed regime, where both endpoints are nonregular at once. Section \ref{sec:examples} illustrates each result, Section \ref{sec:proofs} collects the proofs, and Section \ref{sec:conclusion} concludes the paper.
	
	\section{Preliminaries}\label{prelim}
	
	This section gathers the fundamental material on which the rest of the paper rests. Subsection \ref{notation} fixes the notation and isolates the notion of \emph{essential limit} (Definition \ref{def:limess}), the convention under which every pointwise statement about a function in a Lebesgue space is to be read. Subsection \ref{fam} reviews the elements of the theory of charges required by the multipliers of the nonregular maximum principle and the Yosida--Hewitt decomposition \eqref{YH}, together with the Banach limits (Definition \ref{def:banachlimit}) that will supply generalized limits of multipliers where weak$^*$ sequential compactness fails. The last subsection states problem (P), transcribes the necessary conditions of \cite{be2021} in the form used throughout (Theorem \ref{pmp}), recalls the regularity notion of \cite{dm2}, and records a first consequence of the two: the pure charge can only live where regularity fails (Lemma \ref{lem:chargesupport}).
	
	\subsection{Notation}\label{notation}
	
	The Euclidean norm is represented as $|\cdot |$. Define $\R_+=\{x\in\R : x\geq 0\}$, and for a vector $b\in \R^n$, where $b=(b_1, \ldots, b_n)^T$, we write $b\leq 0$ if $b_i\leq 0$ for all $i=1,\ldots, n$. Similarly, for a vector-valued function $b(t)$, the integral $\int b(t) dt$ denotes the vector $\left(\int b_1(t) dt,\ldots,\int b_n(t) dt\right)^T$.
	
	For a real Banach space $X$, its \textit{dual space} (i.e., the set of bounded linear functionals defined on $X$) is denoted by $X^*$.
	
	Consider a function $h:[0,1]\times\R^n\times \R^m \to \R^q$, differentiable with respect to its last two arguments. The derivative of $h$ with respect to $x$ is denoted by $h_x(t,x,u)$, while its derivative with respect to $u$ is written as $h_u(t,x,u)$.
	
	The space of absolutely continuous functions from $[0,1]$ to $\R^n$ is denoted by $\AC{n}$, the space of integrable functions from $[0,1]$ to $\R^n$ is denoted by $\Lone[n]$, the space of essentially bounded functions is $\Linf[n]$, and $\BV{n}$ denotes the space of functions of bounded variation, which are right-continuous on $(0,1)$. Norms of function spaces are written with double bars and a subscript identifying the space: $\|\cdot\|_{L^1}$, $\|\cdot\|_{L^\infty}$, $\|\cdot\|_{C}$ and, for the total variation norm introduced below, $\|\cdot\|_{TV}$; the single bars $|\cdot|$ are reserved for the Euclidean norm on $\R^n$. The subspace $L_+^\infty([0,1]; \R^n)$ contains functions with nonnegative components, and $L_+^\infty([0,1]; \R^n)^*$ refers to the subspace of $\Linf[n]^*$ consisting of positive functionals, i.e., those mapping elements in $L_+^\infty([0,1]; \R^n)$ to nonnegative numbers. This notation applies analogously to other function spaces.
	
	Given a nominal pair $(\bar x, \bar u) \in \AC{n}\times\Linf[m]$ and a function $k$ depending on $(t,x,u)\in [0,1]\times \R^n\times \R^m$, $\bar k(t)$ denotes $k(t,\bar x(t),\bar u(t))$. 
	
	Given a measurable function $k:[0,1]\to\R$ and a constant $c\in\R$, we abbreviate the level and sublevel sets $\{t\in[0,1]:k(t)=c\}$, $\{t\in[0,1]:k(t)\geq c\}$, $\{t\in[0,1]:k(t)\leq c\}$, $\{t\in[0,1]:k(t)\neq c\}$ by $\{k=c\}$, $\{k\geq c\}$, $\{k\leq c\}$, $\{k\neq c\}$, respectively.
	
	The next concept plays a fundamental role in the sections to come:
	
	\begin{definition}\label{def:limess}
		Let $k:[0,1]\to\R$ be a Lebesgue measurable function and let $t\in[0,1]$. We say that $\alpha\in\R$ is the \emph{essential limit} (or \emph{approximate limit}) of $k$ at $t$, written
		\[
		\limess_{s\to t}k(s)=\alpha,
		\]
		if for every $\varepsilon>0$ there exists a neighborhood $V_\varepsilon$ of $t$ such that $|k(s)-\alpha|<\varepsilon$ for a.e. $s\in V_\varepsilon\cap[0,1]$. When it exists, such an $\alpha$ is unique, and we say that $k$ \emph{admits an essential limit at} $t$.
	\end{definition}
	
	In the sequel, any pointwise equality of the form $k(t)=\alpha$ with $k$ belonging to a Lebesgue space, is shorthand for ``the essential limit of $k$ at $t$ exists and equals $\alpha$.'' We use this notation without fear of ambiguity, as the functions under consideration are defined only up to Lebesgue-null sets.
	
	\subsection{Finitely additive measures and Banach limits}\label{fam}
	Let $\mathfrak{L}$ denote the $\sigma$-algebra of Lebesgue measurable subsets of $[0,1]$, and let $\mathfrak{B}$ be the Borel $\sigma$-algebra on $[0,1]$. A finitely additive set function $\zeta$ defined on $\mathfrak{L}$, such that $\zeta(\emptyset)=0$, is called a \textit{finitely additive measure}, or \textit{charge} for short (see \cite{rao1983}). While every measure on $\mathfrak{L}$ is a charge, the converse is not true, as there are charges that are not measures. We say a charge $\zeta$ is \emph{bounded} if $\sup\{|\zeta(E)| : E \in \mathfrak{L}\} < \infty$. The collection of all bounded charges is denoted by $\text{ba}([0,1],\mathfrak{L})$, and the subset consisting of bounded countably additive charges (i.e., bounded measures) is denoted by $\text{ca}([0,1],\mathfrak{L})$.
	
	We will now introduce some additional definitions regarding charges, which will be useful later (for further properties, see \cite{rao1983}, \cite{dun}, or \cite{yosida}).
	
	\begin{definition}\label{charges}
		Let $\zeta$ and $\lambda$ be charges.
		\begin{enumerate}[label=(\roman*)]
			\item The \emph{total variation} of a charge $\zeta$, denoted $|\zeta|$, is defined as
			\[|\zeta|(E) = \sup \sum_{i=1}^n |\zeta(F_i)|,\]
			where the supremum is taken over all finite partitions $\{F_i\}_{i=1}^n \subset \mathfrak{L}$ of $E$.
			\item We write $\lambda \leq \zeta$ if $\lambda(E) \leq \zeta(E)$ for all $E \in \mathfrak{L}$.
			\item A charge $\lambda$ is said to be \emph{weakly absolutely continuous} with respect to $\zeta$, written $\lambda \ll_w \zeta$, if $|\zeta|(E) = 0$ implies $\lambda(E) = 0$ for any $E \in \mathfrak{L}$.
			\item A positive charge $0 \leq \zeta$ is called \emph{pure} if the only measure $\mu$ such that $0 \leq \mu \leq \zeta$ is the trivial measure $\mu = 0$. A general charge $\zeta$ is called \emph{pure} if its total variation $|\zeta|$ is pure.
			\item A charge $\zeta$ is said to be \emph{right-concentrated at $t^*$} if $\zeta\bigl((t^*, t^*+\varepsilon)\bigr)=\zeta([0,1])$ for all $\varepsilon>0$. It is said to be \emph{left-concentrated at $t^*$} if $\zeta\bigl((t^*-\varepsilon, t^*)\bigr)=\zeta([0,1])$ for all $\varepsilon>0$. 
			\item For a Lebesgue measurable function $f : [0,1] \rightarrow \R$, we say $f = 0$ $\zeta$-\emph{a.e.} if, for every $\varepsilon > 0$,
			\begin{align*}
				|\zeta|(\{|f| > \varepsilon\}) = 0.
			\end{align*}
		\end{enumerate}
	\end{definition}
	
	The space $\text{ba}([0,1],\mathfrak{L})$ is equipped with the \emph{total variation norm}, defined as $\|\zeta\|_{TV} = |\zeta|([0,1])$. It is worth noting that part (vi) of Definition \ref{charges} implies that
	\[
	\zeta \in \text{ba}([0,1], \mathfrak{L}) \text{ and } |\zeta|(\{f \neq 0\}) = 0 ~\Longrightarrow~ f = 0~ \zeta\text{-a.e.},
	\]
	but the converse is not generally true, as shown in \cite[Proposition 4.2.7(ii)]{rao1983}. However, recall that if
	\[
	\mu \in \text{ca}([0,1], \mathfrak{L}) \text{ and } |\mu|(\{f \neq 0\}) = 0, ~\text{then}~ f = 0~ \mu\text{-a.e.}
	\]
	holds in both directions. 
	
	Throughout this paper, unless explicitly stated otherwise, when we write a.e. we refer to the Lebesgue measure, denoted by $\leb$. The space $\text{ba}([0,1],\mathfrak{L},\leb)$ denotes the set of all bounded \emph{weakly absolutely continuous charges} (see Definition \ref{charges}(iii)) with respect to $\leb$, while $\text{ca}([0,1],\mathfrak{L},\leb)$ is the subset of countably additive charges within $\text{ba}([0,1],\mathfrak{L},\leb)$. Since we always consider in this paper absolute continuity in the weak sense of \ref{charges}(iii), we drop the adjective \emph{weak} from now on.
	
	We shall use repeatedly the following two facts. First, the dual space $\Linf^*$ is isometrically isomorphic to $\text{ba}([0,1],\mathfrak{L},\leb)$ (see \cite[Theorem IV.8.16]{dun}). Second, by the Yosida--Hewitt decomposition \cite{yosida}, every $\omega\in\text{ba}([0,1],\mathfrak{L},\leb)$ admits a unique decomposition
	\begin{equation}\label{YH}
		d\omega = z\,d\leb + d\zeta,\qquad z\in \Lone,\qquad \zeta\ \text{a pure charge in } \text{ba}([0,1],\mathfrak{L},\leb),
	\end{equation}
	with $z\geq 0$ a.e. and $\zeta\geq 0$ whenever $\omega\geq 0$ (cf. also \cite[Chapter 10]{rao1983}).
	
	The integration of \emph{bounded measurable functions} with respect to \emph{weakly absolutely continuous charges} with respect to the Lebesgue measure is defined as follows: for a charge $\zeta$, the integral of simple functions is defined in the usual way, and for an arbitrary bounded measurable function $f$, one takes a sequence of simple functions $\psi_n$ converging uniformly a.e. to $f$ and sets
	\[
	\int_0^1 f(t) d\zeta \coloneqq  \lim_{n\rightarrow\infty} \int_0^1 \psi_n(t) d\zeta.
	\]
	For more details on this integral, see \cite[Section 19.3]{roy}, and for more general theories of integration involving charges, see \cite{rao1983} or \cite{dun}.
	
	Finally, regarding notation for integration with respect to a charge: for a function $f \in \Linf{}$ and a charge $\zeta \in \text{ba}([0,1],\mathfrak{L},\leb)$, the expression $d\omega = f d\zeta$ indicates that $\omega \in \text{ba}([0,1],\mathfrak{L},\leb)$ is the charge defined by
	\[
	\omega(E) = \int_E f(t) d\zeta \quad \text{for all} \ E \in \mathfrak{L}.
	\]
	
	Limits of sequences of charges will play a central role in this paper. Since $\Linf$ is not separable, the closed unit ball of its dual, although weak$^*$ compact by the Banach--Alaoglu theorem, is \emph{not} weak$^*$ sequentially compact: bounded sequences of charges need not contain weak$^*$ convergent subsequences (a concrete instance is exhibited in Proposition \ref{prop:limits} below). Compactness must therefore be invoked along subnets or, alternatively, generalized limits may be assigned by means of Banach limits.
	
	\begin{definition}\label{def:banachlimit}
		A \emph{Banach limit} is a linear functional $\blim:\ell^\infty(\mathbb R)\to\R$ satisfying, for every bounded sequence $(x_i)$:
		\begin{enumerate}[label=(\roman*)]
			\item (Positivity) $\blim x_i\geq 0$ whenever $x_i\geq 0$ for all $i$;
			\item (Normality) $\blim (1,1,\ldots)=1$;
			\item (Right-shift invariance) $\blim x_i=\blim y_i$ where $y_i=x_{i+1}$ for all $i$.
		\end{enumerate}
		Banach limits exist as a consequence of the Hahn--Banach Theorem; see, e.g., \cite[pp. 39--41]{rao1983} or \cite[pp. 84--86]{conway1990}. From the previous axioms, it is possible to prove that Banach limits satisfy the following properties:
		\begin{enumerate}[label=(\roman*), resume]
			\item $\liminf{x_i}\le\blim x_i\le\limsup{x_i}$;
			\item $\|\blim\|\coloneqq \sup\left\{ |\blim x_i| : \|(x_i)\|_{\ell^{\infty}}=1 \right\}=1$
		\end{enumerate}
		In particular, if $(x_i)$ is a convergent sequence of real numbers, then any Banach limit is consistent with the usual limit: $\blim x_i=\lim x_i$. On the other hand, Banach limits are not unique and attach a real-value even to nonconvergent sequences in $\ell^\infty(\R)$.
	\end{definition}
	
	\begin{remark}\label{rem:concentrated}
		The existence of concentrated charges can be proved using Banach limits; e.g., a
		concrete realization of a right-concentrated charge at $t^* < 1$ that is absolutely
		continuous with respect to $\leb$ is
		\begin{equation}\label{eq:ccharge}
			\zeta(E) \coloneqq  \blim_i\, i\,\leb\bigl(E \cap (t^*,\, t^* + \tfrac{1}{i})\bigr),
			\qquad E \in \mathfrak{L},
		\end{equation}
		for any Banach limit $\blim$. Finite additivity and positivity follow from the
		linearity and positivity of $\blim$, whereas weak absolute continuity is immediate.
		The fact that $\zeta\bigl((t^*, t^*+\varepsilon)\bigr) = \zeta([0,1]) = 1$ for every $\varepsilon > 0$
		is a consequence of the normality and right-shift invariance axioms. Concentrated charges are necessarily pure; this fact is verified as in the proof of Proposition~\ref{prop:Emult} below.
	\end{remark}
	
	\begin{remark}
		It is important not to confuse the Dirac measure concentrated at a point, say at $t^* = 0$, with a charge $\zeta$ right-concentrated at the same point.
		For the Dirac measure, the singleton $\{0\}$ carries the entire mass; by contrast,
		$\zeta$ satisfies $\zeta(\{0\}) = 0$, and the full charge $\zeta([0,1])$ is instead
		carried by every interval of the form $(0, \varepsilon)$.
	\end{remark}
	
	\subsection{Problem Statement and First Order Optimality Conditions}
	The problem of interest is as follows:
	\begin{equation*}
		\text{(P)}~\left\{	\begin{array}{lll}
			&\text{min}&J(x,u)=\displaystyle\int_0^1L(t,x(t),u(t))dt  \\
			& s.t.&(x,u)\in \AC{n}\times \Linf[m],\\
			&&\dot x(t)=f(t,x(t),u(t))\quad \text{a.e.},\\
			&&g(t,x(t),u(t))\leq 0\ \quad \text{a.e.},\\
			
			&&x(0)=x_0,
		\end{array}\right.
	\end{equation*}
	\noindent
	where $x_0\in\R^n$ is fixed and $(L,f,g):[0,1]\times\R^n\times\R^m\to\R\times\R^n\times\R$ are continuous with respect to $(t,x,u)$ and continuously differentiable with respect to $(x,u)$.
	
	We call any pair $(x,u)\in \AC{n}\times \Linf[m]$ satisfying the constraints of (P) an \emph{admissible process}. Throughout this paper we assume the existence of at least one \emph{weak local minimizer} of (P); i.e., an admissible process $(\bar x,\bar u)$ for which there exists a constant $\varepsilon>0$ such that the inequality $J(\bar x,\bar u)\leq J(x,u)$ holds for all admissible processes $(x,u)$ satisfying
	\[\|x-\bar x\|_C<\varepsilon\quad\text{and}\quad\|u-\bar u\|_{L^\infty}<\varepsilon.\]
	
	Due to the smoothness assumption, we can apply \cite[Theorem 4.1]{be2021} which establishes the following necessary conditions for optimality without assuming regularity:
	
	\begin{theorem}\label{pmp}
		Suppose that $(\bar x,\bar u)$ is a weak local minimizer for (P). Then, there exist a scalar $\lambda_0\geq 0$, a function $z\in \Lone$ and a pure charge $\zeta\in$ ba$([0,1],\mathfrak{L},\leb)$ with $z,\, \zeta\geq 0$, and functions $\lambda,\beta \in \BV{n}$ satisfying the following conditions:
		\begin{enumerate}[label=(\alph*)]
			\item the nontriviality condition
			\[\lambda_0+\|z\|_{L^1}+\zeta([0,1])>0;\]
			\item the complementary slackness conditions
			\[z(t)\bar g(t)=0\ \ \text{a.e.}\ \ \text{and}\ \ \bar g(t)=0\ \ \zeta\text{-a.e.;}\]
			\item the transversality conditions $\lambda(1)=\beta(1)=0$ and the costate equation
			\[-d\lambda=\left[\bar f_x(t)^T\lambda(t)+\lambda_0\bar L_x(t)^T+z(t)\bar g_{x}(t)^T\right]dt+d\beta\]
			where $\beta$ is given by
			\[\beta(0)=-\int_0^1\bar g_x(t)^T\,d\zeta\quad\text{and}\quad\beta(t)=-\lim_{k\to\infty}\int_{t+\frac{1}{k}}^1 \bar g_x(s)^T\, d\zeta \ \text{for}\ t\in (0,1);\]
			\item the stationarity conditions
			\begin{align*}
				0&=\bar f_{u}^T(t)\lambda(t)+\lambda_0\bar L_u^T(t)+\bar g_u^T(t)z(t)\quad\text{a.e.}.\\
				0&=\bar g_{u}^T\,d\zeta.
			\end{align*}
			
		\end{enumerate}
	\end{theorem}
	
	\begin{remark}\label{rem:interval}
		Theorem \ref{pmp} holds, mutatis mutandis, for problems posed on a subinterval $[a,1]\subset[0,1]$ with fixed initial condition $x(a)=x_a$; the multipliers then satisfy $z\in L^1([a,1];\R)$, $\zeta\in\mbox{ba}([a,1],\mathfrak{L},\leb)$, and $\lambda,\beta\in \mathrm{BV}([a,1];\R^n)$. We shall use this version repeatedly when dealing with the auxiliary problems of Sections \ref{sec:when} and \ref{sec:proofs}.
	\end{remark}
	
	Regularity plays a crucial role in optimal control theory, as it simplifies the analysis and strengthens results. As mentioned in the introduction, regularity conditions vary in form, often framed as constraint qualifications or as specific requirements on a nominal process. In this paper, we work with the regularity definition from \cite{dm2} which is as follows:
	
	\begin{definition}\label{reg}
		We say that an admissible process $(\bar x,\bar u)$ is called \emph{regular} if there exist $\varepsilon> 0$ and $v_0\in L^\infty([0, 1]; \R^m)$ such that 
		\[
		\bar g_{u}(t)v_0(t)\geq 1 \quad \text{a.e. in } A_{\varepsilon}\coloneqq \{\bar g\geq-\varepsilon\}.
		\]
	\end{definition}
	
	Under this condition, stronger results can be achieved (see, e.g., \cite{a5,b1,c1,dm2}). Perhaps the most notable advantage is that regularity guarantees the vanishing of the pure charge in Theorem \ref{pmp} and, consequently, the absolute continuity of the costate. In contrast, nonregularity may complicate the analysis due to the intricate structure of the charge space. The following lemma localizes this observation: whatever its global structure, the pure charge can only live where regularity fails. It will be used repeatedly in the sequel.
	
	\begin{lemma}\label{lem:chargesupport}
		Let $(\bar x,\bar u)$ be an admissible process and let $(\lambda,\beta,\lambda_0,z,\zeta)$ satisfy the conditions of Theorem \ref{pmp}. Suppose that $(\bar x,\bar u)$ is regular on a subinterval $I\subseteq[0,1]$; that is, there exist $\varepsilon>0$ and $v_0\in L^\infty(I;\R^m)$ such that $\bar g_u(t)v_0(t)\geq 1$ a.e. on $A_\varepsilon\cap I$. Then $\zeta(E)=0$ for every measurable $E\subseteq I$. In particular, the pure charge is supported on the set where $(\bar x,\bar u)$ fails to be regular.
	\end{lemma}
	
	\begin{proof}
		Extend $v_0$ by zero to $[0,1]$, so that $v_0\in\Linf[m]$ and $\bar g_uv_0\in\Linf$. The second stationarity condition of Theorem \ref{pmp}(d) holds in the charge form established in \cite{be2021}, namely $\int_E\bar g_u\,d\zeta=0$ for every $E\in\mathfrak{L}$; approximating $v_0$ uniformly by simple functions then yields
		\[\int_E\bar g_u(t)v_0(t)\,d\zeta=0\qquad\text{for every }E\in\mathfrak{L}.\]
		Fix a measurable $E\subseteq I$. Since $\bar g\leq 0$ a.e., the set $E\setminus A_\varepsilon=E\cap\{\bar g<-\varepsilon\}$ is contained, up to a Lebesgue null set, in $\{|\bar g|>\varepsilon/2\}$, which is $\zeta$-null by the second complementary slackness condition of Theorem \ref{pmp}(b); as the positive charge $\zeta$ vanishes on Lebesgue null sets, $\zeta(E\setminus A_\varepsilon)=0$, and the integral over $E\setminus A_\varepsilon$ of any function in $\Linf$ vanishes. On $A_\varepsilon\cap I$, in turn, $\bar g_uv_0\geq 1$ holds off a Lebesgue null set, so the positivity and absolute continuity of $\zeta$, together with the monotonicity of the charge integral give
		\[0=\int_E\bar g_u(t)v_0(t)\,d\zeta=\int_{E\cap A_\varepsilon}\bar g_u(t)v_0(t)\,d\zeta\ \geq\ \zeta(E\cap A_\varepsilon)=\zeta(E)\ \geq\ 0,\]
		hence $\zeta(E)=0$.
	\end{proof}
	
	The next section examines specific cases where the lack of regularity results in admissible processes that satisfy the necessary conditions in a trivial manner or result in abnormal multipliers, the latter meaning that $\lambda_{0}=0$.

	\section{When may degeneracy and abnormality phenomena arise?}
	\label{sec:when}
	
	Necessary conditions are useful only insofar as they discriminate among admissible processes, and in the absence of regularity they may fail to do so in two distinct ways. This section identifies the mechanism behind each. We first make precise what a degenerate tuple of multipliers is (Definition \ref{def:trivial}) and exhibit two configurations of the nominal process in which such a tuple is always available (Proposition \ref{prop:trivial}). Subsection \ref{subsec:example} then dissects a model problem, nonregular at the single instant $t=0$, on which Theorem \ref{pmp} degenerates completely, and shows that the degeneration survives every limiting procedure applied to its regular approximations. The final subsection turns to the terminal instant, where nonregularity manifests itself differently: instead of degeneracy, a terminal balance identity (Proposition \ref{prop:termbalance}) tells us whether the problem is completely abnormal or ties the value of the cost multiplier with the value of the charge around the endpoint.
	
	\subsection{Degenerate multipliers}
	
	A key limitation of the PMP is the emergence of degeneracy and abnormality phenomena, which restricts its practical applicability. Degenerate multipliers fail to provide sufficient information to distinguish between admissible processes, while abnormal multipliers fail to relate the cost with the restrictions of the problem. 
	
	In the context of state constraints, a degenerate tuple of multipliers appears when every multiplier vanishes except for the measure associated with the state constraint, which reduces to a Dirac measure concentrated at a single point (typically the initial time $t=0$). The analogue for the general state-control constraints applicable to Theorem \ref{pmp} is formulated as follows:
	
	\begin{definition}\label{def:trivial}
		We say that $(\lambda, \beta, \lambda_0, z, \zeta)$ 
		is a tuple of \emph{degenerate multipliers} if 
		$\lambda_0=0$, $z\equiv 0$, and there exists 
		$t^*\in[0,1]$ such that
		\begin{equation}\label{trivial}
			\lambda(t)=-\beta(t)=\begin{cases}
				\alpha & \mbox{if } t<t^*\\ 0 & \mbox{if } t\ge t^* \end{cases}
		\end{equation}
		where $\zeta$ is a pure charge, absolutely 
		continuous with respect to the Lebesgue measure, 
		and concentrated at $t=t^*$ either from the 
		left or from the right. When $t^*=0$, the values of $\lambda$ and $\beta$ at the single instant $t=0$ are understood to be those prescribed by the formula of Theorem \ref{pmp}(c), namely $\lambda(0)=-\beta(0)=\int_0^1\bar g_x^T(t)\,d\zeta$, which need not vanish.
	\end{definition}
	
	We now address the following question: under what conditions can we find a tuple of degenerate multipliers $(\lambda, \beta, \lambda_0, z, \zeta)$ that satisfies all the conditions of Theorem \ref{pmp}? Proposition \ref{prop:trivial} below identifies two problematic cases when such a behavior occurs.
	
	Recall that from this point forward, whenever an equality of the form $h(t)=\alpha$ is stated for a measurable function $h$, it is understood that $\alpha$ is the essential limit of $h$ at $t$, in the sense of Definition \ref{def:limess}.
	
	The next result pinpoints two scenarios where degeneracy always occurs. We focus on these cases in the forthcoming sections.
	
	\begin{proposition}\label{prop:trivial}
		Suppose that for some $t^*\in [0,1]$, we have $\bar{g}(t^*)=0$ and $\bar{g}_u(t^*)=0$.
		Then, in each of the following cases, there exists a tuple of degenerate multipliers
		satisfying all conditions of Theorem~\ref{pmp}:
		\begin{enumerate}[label=(\roman*)]
			\item $t^*=0$;
			\item $\bar{g}_x(t^*)=0$.
		\end{enumerate}
	\end{proposition}
	
	\begin{proof}
		Suppose first $t^*<1$, let $\zeta$ be a right-concentrated unit charge at $t^*$, and set $\lambda_0=0$ and
		$z\equiv 0$. By Theorem~\ref{pmp}(c), the multiplier $\beta$ must then be given by
		\[
		\beta(t) = \begin{cases}
			-\displaystyle\int_{t^*}^1\bar{g}_{x}^T(s)\,d\zeta & \text{if } t < t^*, \\[6pt]
			0 & \text{if } t \ge t^*,\ t\neq 0,
		\end{cases}
		\]
		together with $\beta(0)=-\int_0^1\bar g_x(s)^T\,d\zeta$ in the case $t^*=0$ (for $t^*>0$ the first branch already covers $t=0$).
		In particular, $\beta$ is a step function, so $d\beta$ is a purely atomic measure
		and hence singular with respect to Lebesgue measure.
		Decomposing the costate equation into its absolutely continuous and
		singular parts yields
		\[
		d\lambda^s = -d\beta, \qquad
		-\dot{\lambda}^a = \bar{f}_x(t)^T\lambda^a(t) \quad \text{a.e.}
		\]
		The transversality condition gives $\lambda(1) = \beta(1) = 0$.
		Since $\lambda^a$ satisfies a homogeneous linear ODE with terminal condition
		$\lambda^a(1) = 0$, uniqueness implies $\lambda^a \equiv 0$.
		Combined with $d\lambda^s = -d\beta$ and $\lambda^s(1) = 0$, this determines
		$\lambda$ uniquely as $\lambda = -\beta$.
		
		At this stage, all conditions of Theorem~\ref{pmp} are satisfied except possibly
		the stationarity condition, which reduces to
		\begin{equation}\label{statlambda}
			\bar{f}_u(t)^T\lambda(t) = 0 \quad \text{a.e. on } [0,1].
		\end{equation}
		It therefore suffices to verify that $\lambda = 0$ a.e. in each case, for then
		\eqref{statlambda} holds trivially.
		In case~(i), since $t^* = 0$, the first branch of the formula is vacuous and
		$\beta$ vanishes on $(0,1]$; the value $\beta(0)=-\int_0^1\bar g_x(s)^T\,d\zeta$
		need not vanish, but is carried by the single instant $t=0$: the atoms of
		$d\lambda$ and $d\beta$ there cancel, by $d\lambda^s=-d\beta$, and
		$\lambda=-\beta$ vanishes a.e., so \eqref{statlambda} holds.
		In case~(ii), since $\zeta$ is a unit charge concentrated at $t^*$ and
		$\bar{g}_x(t^*) = 0$, we have
		$\int_{t^*}^1\bar{g}_x(s)^T\,d\zeta = \bar{g}_x^T(t^*) = 0$,
		so $\beta \equiv 0$ and $\lambda\equiv 0$.
		
		It remains to treat $t^*=1$, which can occur only in case~(ii). A right-concentrated charge at $t^*=1$ does not exist; let instead $\zeta$ be a left-concentrated unit charge at $t^*=1$ (Definition \ref{charges}(v)), as furnished by the mirror image of \eqref{eq:ccharge}, and keep $\lambda_0=0$, $z\equiv 0$. For $t\in(0,1)$ the intervals $(t+\tfrac1k,1]$ eventually contain a left neighborhood of $1$, so the formula of Theorem \ref{pmp}$(c)$ gives $\beta(t)=-\lim_k\int_{t+1/k}^1\bar g_x^T\,d\zeta=-\bar g_x^T(t^*)=0$, and likewise $\beta(0)=-\int_0^1\bar g_x^T\,d\zeta=-\bar g_x^T(t^*)=0$; together with the transversality condition $\beta(1)=0$ this yields $\beta\equiv 0$. The costate argument above then gives $\lambda\equiv 0$, so \eqref{statlambda} holds, while the slackness condition $\bar g=0$ $\zeta$-a.e. and the charge stationarity $\bar g_u\,d\zeta=0$ follow, exactly as in the right-concentrated case, from the essential vanishing of $\bar g$ and $\bar g_u$ at $t^*$ and the concentration of $\zeta$.
		In all cases, the resulting tuple is degenerate in the sense of
		Definition \ref{def:trivial} and satisfies all conditions of Theorem~\ref{pmp}.
	\end{proof}
	
	Proposition \ref{prop:trivial} locates degenerate tuples at a nonregular initial instant, or at a nonregular instant where $\bar g_x$ vanishes. Outside these two cases the picture changes qualitatively. The following example, which we transcribe from \cite[Example 4.3]{be2021} for later use, exhibits a process that is nonregular only at the \emph{terminal} instant $t=1$, with $\bar g_x(1)\neq 0$: there, no choice of multipliers can dispose of the pure charge.
	
	\begin{example}\label{ex:be43}
		Consider the problem
		\begin{equation*}
			\text{(T)}~\left\{	\begin{array}{lll}
				&\text{min}&J(x,u)=\displaystyle\int_0^1 \bigl(u(t)-1\bigr)\,dt,\\
				& s.t.&(x,u)\in \AC{}\times \Linf,\\
				&&\dot x(t)=1-u(t) \quad \text{a.e.},\\
				&&x(t)+u(t)^2-1\leq 0 \quad \text{a.e.},\\
				&&x(0)=0.
			\end{array}\right.
		\end{equation*}
		The process $(\bar x(t),\bar u(t))=(t,0)$ is admissible and is a global minimizer: for any admissible pair $(x,u)$,
		\[\int_0^1\bigl(u(t)-1\bigr)\,dt=-\int_0^1\dot x(t)\,dt=-x(1)\geq -1,\]
		since the mixed constraint gives $x(t)\leq 1-u^2(t)\leq 1$ a.e.\ and hence, by continuity, $x(1)\leq 1$; and $(\bar x,\bar u)$ attains the minimum value $-\bar x(1)=-1$. Here $g(t,x,u)=x+u^2-1$, so that along $(\bar x,\bar u)$
		\[\bar g(t)=t-1,\qquad \bar g_x(t)\equiv 1,\qquad \bar g_u(t)=2\bar u(t)\equiv 0.\]
		The gradient $\bar g_u$ vanishes identically, so Definition \ref{reg} fails; the constraint is active only at the terminal instant, where $\bar g(1)=0$, and the process is nonregular precisely at $t^*=1$, with $\bar g_x(1)=1\neq 0$: neither case of Proposition \ref{prop:trivial} applies.
		
		We claim that \emph{every} tuple $(\lambda,\beta,\lambda_0,z,\zeta)$ satisfying the conditions of Theorem \ref{pmp} for $(\bar x,\bar u)$ has $\zeta\neq 0$, and that $\zeta$ is left-concentrated at $t=1$. Suppose indeed that $\zeta=0$; then $\beta\equiv 0$ by the defining formula in Theorem \ref{pmp}(c), and $\lambda$ is absolutely continuous. Since $\bar g(t)=t-1<0$ for $t<1$, the complementary slackness condition forces $z=0$ a.e. The first stationarity condition then reads
		\[0=\bar f_u(t)\lambda(t)+\lambda_0\bar L_u(t)=-\lambda(t)+\lambda_0\quad\text{a.e.},\]
		while the costate equation reduces to $-\dot\lambda=\bar f_x\lambda+\lambda_0\bar L_x=0$ a.e.; together with the transversality condition $\lambda(1)=0$ these give $\lambda\equiv 0$ and $\lambda_0=0$, so that $\lambda_0+\|z\|_{L^1}+\zeta([0,1])=0$, contradicting the nontriviality condition. Hence $\zeta([0,1])>0$. Moreover, for every $\varepsilon>0$ the set $\{|\bar g|>\varepsilon\}=[0,1-\varepsilon)$ is $\zeta$-null by the second complementary slackness condition, so that, $\zeta$ vanishing on Lebesgue null sets,
		\[\zeta\bigl((1-\varepsilon,1)\bigr)=\zeta([0,1])>0\qquad\text{for every }\varepsilon>0:\]
		the charge is left-concentrated at the terminal instant, in the sense of Definition \ref{charges}(v).
	\end{example}
	
	\begin{remark}\label{rem:be43}
		Example \ref{ex:be43} displays an asymmetry between the initial and the terminal instant which will matter later. First, the unavoidable charge does not entail degeneration. Let $\zeta$ be any positive pure charge in $\ba$, left-concentrated at $t=1$ and of mass $c>0$. Then the tuple
		\[\lambda_0=c,\qquad \lambda=c\,\chi_{[0,1)},\qquad z\equiv 0,\qquad \zeta,\qquad \beta=-c\,\chi_{[0,1)}\]
		satisfies all the conditions of Theorem \ref{pmp}: the formula in (c) gives $\beta(t)=-\zeta\bigl((t,1]\bigr)=-c$ for $t\in[0,1)$ and $\beta(1)=0$; the costate equation reduces, since $\bar f_x=\bar L_x\equiv 0$ and $z\equiv 0$, to the jump identity $-d\lambda=d\beta$ at $t=1$; the first stationarity condition reads $-\lambda(t)+\lambda_0=0$ a.e.\ on $[0,1)$; and the second holds because $\bar g_u\equiv 0$. These multipliers are \emph{normal}. Second, no degenerate tuple exists for (T): in a tuple of the form \eqref{trivial} the stationarity condition forces $\bar f_u(t)^T\lambda(t)=-\lambda(t)=0$ a.e., so $\alpha=0$ and $\lambda=\beta=0$ on $(0,1]$. If $t^*>0$, the formula in Theorem \ref{pmp}(c) then yields $\zeta([0,1])=-\beta(0)=0$; if $t^*=0$, the complementary slackness condition $\bar g=0$ $\zeta$-a.e. is itself incompatible with a nonzero charge right-concentrated at the origin, since $|\bar g(t)|=1-t\geq\tfrac12$ on $[0,\tfrac12]$ makes every set $(0,\varepsilon)$ with $\varepsilon<\tfrac12$ $\zeta$-null. Hence $\zeta=0$ by positivity, violating nontriviality. Thus, at a terminal nonregular instant with $\bar g_x(1)\neq 0$, Theorem \ref{pmp} remains informative even though the charge cannot be nullified, in sharp contrast with the complete degeneration at the initial instant exhibited in Subsection \ref{subsec:example} below. This asymmetry resurfaces in Remark \ref{rem:t4scope2}.
	\end{remark}
	
	The worst possible scenario in the application of Theorem~\ref{pmp} occurs when
	all associated multiplier tuples are degenerate. We formalize this situation in the
	following definition.
	
	\begin{definition}
		A problem is called \emph{completely degenerate} if every tuple of multipliers
		satisfying the conditions of Theorem~\ref{pmp} is degenerate.
	\end{definition}
	
	Before proposing conditions that preclude such situations, we examine in detail an
	example that is completely degenerate, in order to expose the mechanism that
	produces this strong degeneration.
	
	\subsection{Complete degeneration related to initial point nonregularity}
	\label{subsec:example}
	
	In this section we dissect \cite[Example 4.4]{be2021}. The problem is regular except at the single instant $t=0$; nevertheless, we will see that Theorem \ref{pmp} degenerates completely on it, and, more significantly, that the degeneration survives every reasonable limiting procedure applied to regular approximations of the problem. The analysis serves two purposes: it explicitly shows the mechanism responsible for the degeneration, thereby motivating the conditions of Section \ref{sec:nondeg}, and it issues a methodological warning concerning compactness arguments for multipliers in $\Linf^*$.
	
	\begin{example}\label{ex:E}
		Consider the problem
		\begin{equation*}
			\text{(E)}~\left\{	\begin{array}{lll}
				&\text{min}&J(x,u)=-\displaystyle\int_0^1 x(t)dt,\\
				& s.t.&(x,u)\in \AC{}\times \Linf,\\
				&&\dot x(t)=u(t) \quad \text{a.e.},\\
				&&tu(t)-t^2\leq 0 \quad \text{a.e.},\\
				&&x(0)=0.
			\end{array}\right.
		\end{equation*}
		Here $g(t,x,u)=tu-t^2$, so the mixed constraint amounts to $u(t)\leq t$ for a.e. $t\in(0,1]$, and imposes no restriction at $t=0$. Consequently, every admissible trajectory satisfies $x(t)=\int_0^tu(s)ds\leq t^2/2$, hence $J(x,u)\geq -1/6$, with equality if and only if $u(t)=t$ a.e. Thus $(\bar x,\bar u)=(t^2/2,\,t)$ is the unique global minimizer of (E). Along it,
		\[\bar g(t)\equiv 0,\qquad \bar g_u(t)=t,\qquad \bar g_x(t)\equiv 0,\]
		so the constraint is active on all of $[0,1]$ and $(\bar x,\bar u)$ is nonregular: $A_\varepsilon=[0,1]$ for every $\varepsilon>0$, and any $v_0$ with $t\,v_0(t)\geq 1$ a.e. is unbounded near the origin. Observe, however, that the nonregularity is confined to $t=0$: Definition \ref{reg} is satisfied on every subinterval $[a,1]$ with $a>0$.
	\end{example}
	
	Our first observation is that the multipliers of (E) can be computed exactly, and that they are degenerate in the strongest possible sense: all the multipliers satisfying the conditions of Theorem \ref{pmp} are degenerate.
	
	\begin{proposition}\label{prop:Emult}
		A tuple $(\lambda, \beta,\lambda_0, z, \zeta)$ satisfies conditions~(a)--(d)
		of Theorem~\ref{pmp} for $(\bar{x}, \bar{u})$ in Example~\ref{ex:E} if and
		only if it is a tuple of degenerate multipliers with $\lambda \equiv \beta \equiv 0$
		and $\zeta$ concentrated at the origin. In particular, Theorem~\ref{pmp} does
		not distinguish among admissible processes of~$(E)$: the problem is completely
		degenerate.
	\end{proposition}
	
	\begin{proof}
		Since $g$ does not depend on $x$, we have $\bar{g}_x \equiv 0$ along any
		admissible process; hence $\beta \equiv 0$, and $\lambda$ is absolutely continuous.
		
		\emph{Necessity.} Here $f(t,x,u) = u$ and $L(t,x,u) = -x$, so
		$\bar{f}_x \equiv 0$ and $\bar{L}_x \equiv -1$. The costate equation reads
		$\dot{\lambda}(t) = \lambda_0$ a.e., which together with $\lambda(1) = 0$ gives
		$\lambda(t) = \lambda_0(t-1)$. The first stationarity condition in~(d) becomes
		$\lambda(t) + t\,z(t) = 0$ a.e., i.e., $z(t) = \lambda_0(1-t)/t$ a.e.
		If $\lambda_0 > 0$, then $z \notin \Lone$, which is impossible; hence
		$\lambda_0 = 0$, $\lambda \equiv 0$, and $z = 0$ a.e., and the nontriviality
		condition~(a) forces $\zeta([0,1]) > 0$. The second
		stationarity condition in~(d), $t\,d\zeta=0$, reads
		$\int_E t\,d\zeta = 0$ for all $E \in \mathfrak{L}$, and for a positive
		charge this is equivalent to concentration at the origin: choosing
		$E = [\varepsilon, 1]$ yields $0 = \int_E t\,d\zeta \geq \varepsilon\,\zeta(E)$,
		so $\zeta([\varepsilon, 1]) = 0$. Lastly, the complementary slackness
		conditions~(b) hold trivially since $\bar{g} \equiv 0$.
		
		\emph{Purity.} Let $\mu$ be a measure with $0 \leq \mu \leq \zeta$. Then
		$\mu\bigl((1/k, 1]\bigr) \leq \zeta\bigl((1/k, 1]\bigr) = 0$ for every $k$, and
		$\mu(\{0\}) \leq \zeta(\{0\}) = 0$ because $\zeta$ is weakly absolutely continuous
		and $\leb(\{0\}) = 0$. By countable additivity,
		$\mu([0,1]) = \mu(\{0\}) + \lim_k \mu\bigl((1/k,1]\bigr) = 0$,
		so $\mu = 0$ and $\zeta$ is pure.
		
		\emph{Sufficiency along an arbitrary admissible process $(x,u)$.} With
		$\lambda_0 = 0$, $z = 0$, and $\bar{g}_x \equiv 0$, condition~(c) yields
		$\lambda \equiv \beta \equiv 0$, and the first condition in~(d) holds trivially.
		The second condition in~(d), that $\int_E t\,d\zeta = 0$ for every
		$E\in\mathfrak L$, was verified above and does
		not involve $(x,u)$. As for~(b), the first condition holds trivially since
		$z \equiv 0$; for the second condition, observe that
		$|g(t,x(t),u(t))| = |tu(t) - t^2| \leq t(\|u\|_{L^\infty} + 1)$, so for every
		$\varepsilon > 0$ the set $\{t : |g(t,x(t),u(t))| > \varepsilon\}$ is contained in
		$\{t : t > \varepsilon/(\|u\|_{L^\infty}+1)\}$, which is $\zeta$-null; hence
		$g(t,x(t),u(t)) = 0$ $\zeta$-a.e.
	\end{proof}
	
	Since the nonregularity of (E) is confined to $t=0$, it is natural to attempt to recover informative multipliers by approximating (E) with the regular problems obtained by truncating the interval. For $n\geq 2$, consider
	\begin{equation*}
		\text{(E$_n$)}~\left\{	\begin{array}{lll}
			&\text{min}&\displaystyle-\int_{1/n}^1 x(t)dt,\\
			& s.t.&(x,u)\in W^{1,1}([\tfrac1n,1];\R)\times L^\infty([\tfrac1n,1];\R),\\
			&&\dot x(t)=u(t) \quad \text{a.e. on } [\tfrac1n,1],\\
			&&tu(t)-t^2\leq 0 \quad \text{a.e.},\\
			&&x(\tfrac1n)=\tfrac{1}{2n^2}.
		\end{array}\right.
	\end{equation*}
	The restriction of $(\bar x,\bar u)$ to $[1/n,1]$ is the unique global minimizer of (E$_n$), by the same argument as in Example \ref{ex:E}, and it is now \emph{regular}: the constant function $v_0\equiv n$ yields $\bar g_u(t)v_0=nt\geq 1$ on $[1/n,1]$. The next lemma shows that each (E$_n$) is as well behaved as one could wish: its multipliers are unique up to a positive scalar, normal, and free of charges.
	
	\begin{lemma}\label{lem:En}
		The tuples satisfying conditions (a)--(d) of Theorem \ref{pmp} (transcribed to $[1/n,1]$; see Remark \ref{rem:interval}) for (E$_n$) at $(\bar x,\bar u)|_{[1/n,1]}$ are exactly the positive scalar multiples of
		\[\lambda_0^n=1,\qquad \lambda_n(t)=t-1,\qquad z_n(t)=\frac{1-t}{t},\qquad \zeta_n=0,\qquad \beta_n\equiv 0.\]
	\end{lemma}
	
	\begin{proof}
		As before, $\bar g_x\equiv 0$ forces $\beta\equiv 0$ and $\lambda$ absolutely continuous. The second stationarity condition gives $\int_E t\,d\zeta_n=0$ for all measurable $E\subseteq[1/n,1]$; taking $E=[1/n,1]$ and using $t\geq 1/n$ there, $0=\int_E t\,d\zeta_n\geq \tfrac1n\zeta_n([\tfrac1n,1])$, hence $\zeta_n=0$. The costate equation and $\lambda(1)=0$ give $\lambda(t)=\lambda_0(t-1)$, and the first stationarity condition gives $z(t)=\lambda_0(1-t)/t$, which now belongs to $L^1([1/n,1];\R)$. If $\lambda_0=0$, then $\lambda\equiv 0$ and $z\equiv 0$, violating nontriviality; hence $\lambda_0>0$ and we may scale $\lambda_0=1$.
	\end{proof}
	
	The pathology emerges only in the limit $n\to\infty$. Normalize the multipliers of (E$_n$) without relabeling, so that
	\begin{equation}\label{normEn}
		\lambda_0^n+\|z_n\|_{L^1}=1,
	\end{equation}
	that is, divide the tuple of Lemma \ref{lem:En} by $1+a_n$, where
	\[a_n\coloneqq \int_{1/n}^1\frac{1-t}{t}\,dt=\ln n-1+\frac1n.\]
	Extend $z_n$ by zero to $[0,1]$ and let $\mu_n\in\text{ca}_+([0,1],\mathfrak{L},\leb)$ denote the associated measure $d\mu_n = z_n(t)dt$.
	
	\begin{proposition}\label{prop:limits}
		With the normalization \eqref{normEn},
		\[\lambda_0^n=\Big(\ln n+\frac1n\Big)^{-1}\longrightarrow 0,\qquad \|\lambda_n\|_{L^\infty}=\lambda_0^n\Big(1-\frac1n\Big)\longrightarrow 0,\qquad \|z_n\|_{L^1}\longrightarrow 1,\]
		while, for every $\varepsilon\in(0,1)$, $\mu_n([\varepsilon,1])\leq \lambda_0^n\ln(1/\varepsilon)\to 0$: the mass of $z_n$ drifts towards the nonregular instant. Moreover:
		\begin{enumerate}[label=(\roman*)]
			\item no subsequence of $(z_n)$ converges weakly in $\Lone$;
			\item regarded as Borel measures, $\mu_n\rightharpoonup\delta_0$ in the weak$^*$ topology of $\C{}^*$, where $\delta_0$ is the Dirac measure at the origin; however, $\delta_0$ fails to be weakly absolutely continuous with respect to $\leb$ and is therefore not an admissible multiplier;
			\item the sequence $(\mu_n)$ has no limit in the weak$^*$ topology of $\Linf^*$: for $\psi(t)=\sin(\ln\ln(e/t))$, $t\in(0,1]$, the sequence $\big(\int_0^1\psi\,d\mu_n\big)$ clusters at every point of $[-\tfrac{\sqrt2}{2},\tfrac{\sqrt2}{2}]$;
			\item for every Banach limit $\blim$, the set function
			\[\omega(E)\coloneqq \blim_n\ \mu_n(E),\qquad E\in\mathfrak{L},\]
			is a positive pure charge in $\mbox{\emph{ba}}([0,1],\mathfrak{L},\leb)$, concentrated at the origin and with $\omega([0,1])=1$. Every weak$^*$ cluster point of $(\mu_n)$ in $\Linf^*$ enjoys the same properties.
		\end{enumerate}
		Consequently, every generalized limit of the normalized multipliers of (E$_n$), whether a Banach limit or a weak$^*$ cluster point along a subnet, is a tuple of degenerate multipliers.
	\end{proposition}
	
	\begin{proof}
		The normalization \eqref{normEn} gives $\lambda_0^n(1+a_n)=1$, i.e., $\lambda_0^n=(\ln n+1/n)^{-1}$; the expressions for $\|\lambda_n\|_{L^\infty}$ and $\|z_n\|_{L^1}=1-\lambda_0^n$ are immediate, and for $n>1/\varepsilon$,
		\[\mu_n([\varepsilon,1])=\lambda_0^n\int_\varepsilon^1\frac{1-t}{t}dt\leq\lambda_0^n\int_\varepsilon^1\frac{dt}{t}=\lambda_0^n\ln\frac1\varepsilon.\]
		
		(i) Suppose $z_{n_k}\rightharpoonup w$ weakly in $L^1$. Testing against $\chi_{[\varepsilon,1]}\in L^\infty$ yields $\int_\varepsilon^1w\,dt=\lim_k\mu_{n_k}([\varepsilon,1])=0$ for every $\varepsilon>0$, so $w=0$ a.e.; testing against $\mathbf 1$ yields $\int_0^1 w\,dt=\lim_k\|z_{n_k}\|_{L^1}=1$, a contradiction.
		
		(ii) Let $\varphi\in \C{}$ and $\delta>0$. Then
		\[\Big|\int_0^1\varphi\,d\mu_n-\varphi(0)\|z_n\|_{L^1}\Big|\leq\sup_{[0,\delta]}|\varphi-\varphi(0)|\,\|z_n\|_{L^1}+2\|\varphi\|_{L^\infty}\,\mu_n([\delta,1]).\]
		Letting $n\to\infty$ and then $\delta\to 0$, and using $\|z_n\|_{L^1}\to1$, we get $\int_0^1\varphi\,d\mu_n\to\varphi(0)$.
		
		(iii) Clearly $\psi\in L^\infty$ with $\|\psi\|_{L^\infty}\leq 1$. Write
		\[\int_0^1\psi\,d\mu_n=\lambda_0^n\int_{1/n}^1\psi(t)\frac{dt}{t}-\lambda_0^n\int_{1/n}^1\psi(t)\,dt,\]
		and note that the second term is $O(\lambda_0^n)$. The substitution $s=\ln(e/t)$ turns the first integral into
		\[\int_1^{1+\ln n}\sin(\ln s)\,ds=\Big[\frac{s}{2}\big(\sin\ln s-\cos\ln s\big)\Big]_1^{1+\ln n}=\frac{(1+\ln n)\sqrt2}{2}\,\sin\Big(\theta_n-\frac\pi4\Big)+\frac12,\]
		where $\theta_n\coloneqq \ln(1+\ln n)$. Since $\lambda_0^n(1+\ln n)\to 1$, we obtain
		\[\int_0^1\psi\,d\mu_n=\frac{\sqrt2}{2}\sin\Big(\theta_n-\frac\pi4\Big)+o(1).\]
		The sequence $(\theta_n)$ increases to infinity with $\theta_{n+1}-\theta_n\to 0$, so its image modulo $2\pi$ is dense in $[0,2\pi]$; the cluster set of $\big(\int\psi\,d\mu_n\big)$ is therefore the full interval $[-\tfrac{\sqrt2}{2},\tfrac{\sqrt2}{2}]$, and the sequence does not converge.
		
		(iv) Finite additivity of $\omega$ follows from the linearity of $\blim$ and the additivity of each $\mu_n$, with $\omega(\emptyset)=0$; positivity and the bound $0\leq\omega(E)\leq\sup_n\|z_n\|_{L^1}\leq 1$ follow from the positivity and consistency of $\blim$. If $\leb(E)=0$, then $\mu_n(E)=0$ for all $n$ and $\omega(E)=0$, so $\omega\in\mbox{ba}([0,1],\mathfrak{L},\leb)$. Consistency implies as well that $\omega([\varepsilon, 1])=\blim_n\mu_n([\varepsilon,1])=0$ for every $\varepsilon>0$ and $\omega([0,1])=\blim_n(1-\lambda_0^n)=1$. Purity now follows exactly as in the proof of Proposition \ref{prop:Emult}.
		
		Finally, if $\omega'$ is a weak$^*$ cluster point of $(\mu_n)$ along a subnet, testing against $\mathbf 1$, $\chi_{[\varepsilon,1]}$ and $\chi_E$ with $\leb(E)=0$ shows that $\omega'$ has total mass 1, is concentrated at the origin and is weakly absolutely continuous, hence pure by the same argument.
		
		For the final claim, recall that $\lambda_0^n\to 0$, $\lambda_n\to 0$ uniformly, and $\zeta_n=0$, $\beta_n\equiv 0$ for every $n$; thus every generalized limit of the tuple $(\lambda_0^n,z_n,\zeta_n,\lambda_n,\beta_n)$ is of the form $(0,0,\omega,0,0)$ with $\omega$ as above, concluding the proof.
	\end{proof}
	
	\begin{remark}\label{rem:mech}
		The computations above expose the mechanism of degeneration. The stationarity
		condition pins the integrable multiplier to
		$z_n(t) = \lambda_0^n(1-t)/\bar{g}_u(t)$ with $\bar{g}_u(t) = t$, and
		$1/\bar{g}_u$ is \emph{not} integrable near the nonregular instant. Any
		normalization must therefore crush the cost multiplier at a logarithmic rate and
		transfer the entire mass to $z_n$, pushing it toward $t = 0$. In the limit, this
		mass can only be represented by a pure charge sitting ``immediately to the right
		of'' the origin: an atom at $\{0\}$ itself is forbidden by weak absolute
		continuity, in contrast with the state-constraint setting, where degeneration
		materializes as an atom of the measure multiplier at the initial time
		\cite{ferreira,R-V}. Note that the limiting charge is invisible to every test
		function vanishing near the origin: $\blim_n\int\psi\,d\mu_n = 0$ whenever
		$\psi \in L^\infty$ vanishes a.e.\ on a neighborhood of $0$; the only surviving
		feature is the total mass, lodged at the nonregular instant, which carries no
		information about the minimizer. This location of the mass is an instance of
		the support principle of Lemma \ref{lem:chargesupport}: the process being
		regular on every $[\delta,1]$, any surviving pure charge must vanish there,
		and can therefore only sit at the nonregular instant $t=0$.
		
		Two structural features of~$(E)$ make this escape of mass possible: (a) the
		constraint is active on a full neighborhood of the nonregular instant; and
		(b)~$\bar{g}_x \equiv 0$, so the costate equation exerts no control whatsoever
		on the accumulation of $z_n$ near $t = 0$.
	\end{remark}
	
	\begin{remark}\label{rem:method}
		Items~(i)--(iv) of Proposition~\ref{prop:limits} raise several points worth
		emphasizing for the proofs of Section~\ref{sec:nondeg}. Bounded sequences of
		multipliers in $\Linf^*$ need not contain weak$^*$ convergent
		subsequences; item~(iii) exhibits this concretely, so compactness must
		be invoked either through the Banach--Alaoglu theorem along subnets, or, as in
		item~(iv), through Banach limits, which assign to every bounded sequence a
		canonical generalized limit. The Yosida--Hewitt decomposition~\eqref{YH} then
		splits any such limit $\omega$ uniquely as $\omega = z\,d\leb + \zeta$ with
		$z \in L^1$ and $\zeta$ pure, recovering the structure of the multipliers of
		Theorem~\ref{pmp}. Furthermore, the example shows that approximating a problem
		by auxiliary problems, each of them regular in the sense of Definition \ref{reg} and having nondegenerate multipliers, is no guarantee of nondegeneracy in the limit. When, on the contrary, the approximating densities admit a fixed integrable majorant, even weak $L^1$ compactness (the Dunford--Pettis theorem) becomes unnecessary: the domination alone renders the limit charge countably additive and annihilates its pure part, as in the proofs of Theorems \ref{t1} and \ref{t2} below.
	\end{remark} 
	
	\subsection{Abnormality related to end point nonregularity}
	
	The previous study identifies the lack of integrability of $|\bar g_u|^{-1}$ near an initial nonregular instant as the main trigger for complete degeneracy. The next proposition evidences an asymmetry between initial and final nonregular instants; nonintegrability around the latter may trigger either \emph{complete abnormality}, or force every normal tuple to carry a \emph{nonvanishing pure charge}.
	\begin{proposition}\label{prop:termbalance}
		Let $(\bar x,\bar u)$ be a weak local minimizer for (P) such that, for some $\delta_0\in(0,1)$,
		\begin{equation}\label{nonintterm}
			\int_{1-\delta}^1\frac{dt}{|\bar g_u(t)|}=\infty\qquad\text{for every }\delta\in(0,\delta_0],
		\end{equation}
		with the convention $1/0=+\infty$, and suppose that each entry of $\bar f_u$, $\bar L_u$ and $\bar g_x$ admits an essential limit at $t=1$. Let $(\lambda,\beta,\lambda_0,z,\zeta)$ be any tuple satisfying the conditions of Theorem \ref{pmp}, and set
		\[\zeta_1\coloneqq\lim_{t\to 1^-}\zeta\bigl((t,1]\bigr)\in\bigl[0,\zeta([0,1])\bigr].\]
		Then the costate has left limit $\lambda(1^-)=\zeta_1\,\bar g_x(1)^T$, and the \emph{terminal balance identity}
		\begin{equation}\label{balance}
			\lambda_0\,\bar L_u(1)^T+\zeta_1\,\bar f_u(1)^T\bar g_x(1)^T=0
		\end{equation}
		holds. In particular:
		\begin{enumerate}[label=(\roman*)]
			\item if $\bar L_u(1)\neq 0$ and $\bar L_u(1)^T\neq-c\,\bar f_u(1)^T\bar g_x(1)^T$ for every $c>0$, then $\lambda_0=0$; if moreover $\bar f_u(1)^T\bar g_x(1)^T\neq 0$, then also $\zeta_1=0$;
			\item if $\bar L_u(1)^T=-c_0\,\bar f_u(1)^T\bar g_x(1)^T$ for some $c_0>0$ and $\bar f_u(1)^T\bar g_x(1)^T\neq 0$, then $\zeta_1=c_0\lambda_0$: every normal tuple carries a pure charge placing mass at least $c_0\lambda_0>0$ on every interval $(1-\varepsilon,1)$.
		\end{enumerate}
	\end{proposition}
	
	\begin{proof}
		Write $h(t)\coloneqq\bar f_u(t)^T\lambda(t)+\lambda_0\bar L_u(t)^T\in\R^m$. The first stationarity condition of Theorem \ref{pmp}(d) reads $h(t)=-\bar g_u(t)^Tz(t)$ a.e., and since $z\geq 0$ a.e.,
		\begin{equation}\label{hz}
			|h(t)|=z(t)\,|\bar g_u(t)|\qquad\text{a.e.}
		\end{equation}
		
		On the other hand, being of bounded variation, $\lambda$ possesses a left limit $\lambda(1^-)$ which coincides with its essential limit; together with the essential limits of $\bar f_u$ and $\bar L_u$ at $1$, this shows that $h$ admits the essential limit
		\[h(1)=\bar f_u(1)^T\lambda(1^-)+\lambda_0\,\bar L_u(1)^T\]
		at $t=1$ (componentwise, in the sense of Definition \ref{def:limess}). Suppose $\rho\coloneqq|h(1)|>0$. Then there is $\delta\in(0,\delta_0]$ with $|h(t)|\geq\rho/2$ a.e. on $(1-\delta,1)$; at almost every such $t$, \eqref{hz} forces $|\bar g_u(t)|>0$ together with
		\[z(t)=\frac{|h(t)|}{|\bar g_u(t)|}\geq\frac{\rho}{2\,|\bar g_u(t)|},\]
		so \eqref{nonintterm} gives $\int_{1-\delta}^1z\,dt=\infty$, contradicting $z\in\Lone$. Hence $h(1)=0$.
		
		In view of the fact that the function $t\mapsto\zeta((t,1])$ is non-increasing by the positivity and finite additivity of $\zeta$, the limit defining $\zeta_1$ exists. Moreover, $\lambda+\beta$ coincides a.e. with an absolutely continuous function $p$ with $p(1)=0$, while $\beta$ coincides a.e. with the function $t\mapsto-\int_{(t,1]}\bar g_x(s)^T\,d\zeta$, by Lemma \ref{lem:betacharge} below. It follows that
		\[\lambda(1^-)=p(1)-\lim_{t\to1^-}\Bigl(-\int_{(t,1]}\bar g_x(s)^T\,d\zeta\Bigr)=\lim_{t\to1^-}\int_{(t,1]}\bar g_x(s)^T\,d\zeta.\]
		Fix $\varepsilon>0$ and choose $\delta>0$ with $|\bar g_x(s)-\bar g_x(1)|\leq\varepsilon$ a.e. on $(1-\delta,1)$. Since $\zeta$ vanishes on Lebesgue null sets, uniform approximation by simple functions gives, for every $t\in(1-\delta,1)$,
		\[\Bigl|\int_{(t,1]}\bar g_x(s)^T\,d\zeta-\bar g_x(1)^T\,\zeta\bigl((t,1]\bigr)\Bigr|=\Bigl|\int_{(t,1]}\bigl(\bar g_x(s)-\bar g_x(1)\bigr)^T\,d\zeta\Bigr|\leq\varepsilon\,\zeta\bigl((t,1]\bigr)\leq\varepsilon\,\zeta([0,1]).\]
		Letting $t\to1^-$ and then $\varepsilon\to0^+$ yields $\lambda(1^-)=\zeta_1\,\bar g_x(1)^T$; substituting into $h(1)=0$ produces \eqref{balance}.
		
		Now, we examine the cases in (i) and (ii). In case (i), if $\lambda_0>0$, then \eqref{balance} gives $\zeta_1\,\bar f_u(1)^T\bar g_x(1)^T=-\lambda_0\,\bar L_u(1)^T\neq 0$, so $\zeta_1>0$ and $\bar L_u(1)^T=-(\zeta_1/\lambda_0)\,\bar f_u(1)^T\bar g_x(1)^T$, which is excluded; hence $\lambda_0=0$, and then $\zeta_1\,\bar f_u(1)^T\bar g_x(1)^T=0$. Furthermore, $\zeta_1=0$ whenever $\bar f_u(1)^T\bar g_x(1)^T\neq 0$. In case (ii), substituting $\bar L_u(1)^T=-c_0\,\bar f_u(1)^T\bar g_x(1)^T$ into \eqref{balance} yields $(\zeta_1-c_0\lambda_0)\,\bar f_u(1)^T\bar g_x(1)^T=0$, hence $\zeta_1=c_0\lambda_0$.
	\end{proof}
	
	\begin{remark}\label{rem:termbalancea}
		
		Proposition \ref{prop:termbalance} admits a variant free of any integrability requirement: if the constraint is inactive a.e. on some interval $(1-\eta,1)$, as in problem (T) in Example~\ref{ex:be43}, then the complementary slackness condition gives $z=0$ a.e. there, the stationarity condition yields $h=0$ a.e. near $1$ and the essential limit $h(1)=0$ holds without \eqref{nonintterm}. For (T) one has $\bar L_u\equiv 1$, $\bar f_u\equiv-1$ and $\bar g_x\equiv 1$ (here \eqref{nonintterm} does hold anyway, since $\bar g_u\equiv 0$), so case (ii) applies with $c_0=1$: \emph{every} tuple of (T) satisfies $\zeta_1=\lambda_0$. This sharpens the findings of Example \ref{ex:be43} provided in \cite{be2021}.
	\end{remark}
	
	\begin{remark}\label{rem:termbalanceb}	
		Hypothesis \eqref{nonintterm} cannot be dropped. In the integrable terminal regime of Example \ref{ex:terminalint} the tuples \eqref{Etermfamily} have $h(t)=\lambda(t)+2\lambda_0$, with essential limit $C=2\lambda_0-m$ at $t=1$: nonzero whenever $m<2\lambda_0$. Accordingly $\zeta_1=m$ sweeps the whole interval $[0,2\lambda_0]$ and no identity ties the terminal mass to the cost multiplier: the integrability of $|\bar g_u|^{-1}$ is exactly the condition that ensures the terminal mass decouples from $\lambda_0$.
	\end{remark}
	
	\section{Nondegeneracy and normality conditions}\label{sec:nondeg}
	
	In this section, we propose several conditions to ensure nondegeneracy and obtain further desirable properties of the multipliers. Conditions \textbf{(C1)}--\textbf{(C3)} address a nonregular \emph{initial} instant by approximation with regular truncations, and yield multipliers that are nondegenerate, free of charges and, under an inward-pointing requirement, normal (Theorems \ref{t1}, \ref{t2} and \ref{t3}). Condition \textbf{(C4)} dispenses with approximation altogether and forces \emph{every} tuple to be normal, together with a bound on the mass of its charge (Theorem \ref{t4}); by Remark \ref{rem:t4scope1} it is unavailable at the origin, so the two groups of results complement rather than subsume one another. The conditions in Subsection \ref{subsec:mixed} bridge these cases, covering processes that are nonregular at both endpoints at once (Theorem \ref{t5}).
	
	\subsection{Nonregularity at the initial instant under regular approximations}
	
	Throughout this subsection, we suppose that we can approximate by auxiliary \emph{regular} problems:
	
	\begin{itemize}
		\item[\textbf{(R)}] for every $\delta\in(0,1)$ the restriction of $(\bar x,\bar u)$ to $[\delta,1]$ is regular in the sense of Definition~ \ref{reg}
	\end{itemize}
	
	The conditions on the next result not only guarantee the existence of a normal tuple of multipliers; they ensure as well that the charge vanishes, which in turn implies the absolute continuity of the costate.
	
	\begin{theorem}\label{t1}
		Assume \textbf{(R)} and let $(\bar x,\bar u)$ be a weak local minimizer for (P) satisfying $\bar g(0)=0$ and $\bar g_u(0)=0$. Suppose the following condition holds:
		\begin{itemize}
			\item[\textbf{(C1)}] there exists $v\in\Lone[m]$ such that, for a.e. $t\in\{\bar g=0\}$,
			\[\bar f_u(t)\,v(t)=0\qquad\text{and}\qquad \langle\bar g_u(t),\, v(t)\rangle\ge 1.\]
		\end{itemize}
		Then there exists a tuple of nondegenerate multipliers $(\lambda,\beta,\lambda_0,z,\zeta)$ satisfying the conditions of Theorem \ref{pmp}  with $\zeta=0$, $\beta=0$, and whose cost multiplier is strictly positive; in fact,
		\[\lambda_0\ \ge\ \frac{1}{1+\|\bar L_u\|_{L^\infty}\,\|v\|_{L^1}}\ >\ 0.\]
		In particular, the costate $\lambda$ is absolutely continuous and after normalization one may take $\lambda_0=1$.
	\end{theorem}
	
	\begin{remark}\label{rem:c1active}
		In Theorem \ref{t1} it suffices to impose \textbf{(C1)} on the active set $\{\bar g=0\}$, rather than on the larger near-active set $A_\varepsilon$ of Definition \ref{reg}: by the complementary slackness condition, the integrable multipliers of the auxiliary problems appearing in the proof are supported on $\{\bar g=0\}$, so the pointwise bound \eqref{star} is unaffected, while the near-active regularity is supplied separately by hypothesis \textbf{(R)}, which remains in force. The analogous weakening is \emph{not} available in Theorem \ref{t2} below: there the lower bound of \textbf{(C2)} on all of $A_\varepsilon$ is what renders each auxiliary problem regular. Were the bound required only on $\{\bar g=0\}$, the gradient $\bar g_u$ could degenerate on $A_\varepsilon\setminus\{\bar g=0\}$, the auxiliary problems could fail to be regular, and their multipliers could carry pure charges which the domination argument does not remove.
	\end{remark}
	
	\begin{remark}
		In \cite[Example 4.3]{be2021} (transcribed as Example \ref{ex:be43} above) it is shown that, if the constraint and its gradient with respect to \( u \) vanish simultaneously only at the final point \( t = 1 \) instead of the initial point \( t = 0 \), as assumed earlier, then the nullification of the charge is generally not possible, even if the constraint is not activated at any other point within the interval. As Remark \ref{rem:be43} shows, normality is nevertheless retained there.
	\end{remark}
	
	Next we relax the hypotheses of Theorem \ref{t1}, expecting, of course, a weaker conclusion: nondegenerate, rather than normal, multipliers. We do so by replacing \textbf{(C1)} with an integrable lower bound on $|\bar g_u|$ near the nonregular instant; in particular, we drop the requirement $\bar f_u v=0$ from \textbf{(C1)}. The price we pay is precisely the loss of the uniform lower bound on the cost multiplier. Nonetheless, the annihilation of the charge is preserved.
	
	\begin{theorem}\label{t2}
		Let $(\bar x, \bar u)$ be a weak local minimizer for (P) such that $\bar g(0)=0$ and $\bar g_u(0)=0$. If the following condition holds:
		\begin{itemize}
			\item[\textbf{(C2)}] there exist $\varepsilon>0$ and a non-decreasing function $\omega:(0,1]\to(0,\infty)$ with
			\[\int_0^1\frac{dt}{\omega(t)}<\infty\]
			such that, for a.e. $t\in \{\bar g\geq-\varepsilon\}$
			\[|\bar g_u(t)|\geq\omega(t),\]
		\end{itemize}
		then there exists a tuple of multipliers $(\lambda,\beta,\lambda_0,z,\zeta)$ satisfying the conditions of Theorem \ref{pmp} with $\zeta=0$ and $\beta\equiv 0$; in particular, the multipliers are nondegenerate.
	\end{theorem}
	
	\begin{remark}
		Note that, although we are not explicitly assuming \textbf{(R)}, this condition is implied by \textbf{(C2)} (see the details in Section~\ref{subsec:prooft2}). Moreover, this condition excludes, for instance, lower bounds of the form $|\bar g_u(t)|\geq c\,t$ (which would force $\omega(t)\leq ct$ and hence $\int_0^1\omega^{-1}=\infty$), precisely the regime in which the degeneracy phenomenon persists; see for example Subsection \ref{subsec:example}. It does, however, admit Hölder-type irregularities $|\bar g_u(t)|\geq c\,t^{\alpha}$ with $0\le\alpha<1$.
	\end{remark}
	
	\begin{remark}
		The essential requirement in Theorem \ref{t2} is the nondecreasing integrable lower bound $|\bar g_u(t)|\geq\omega(t)$ near $t=0$ with $1/\omega\in L^1$: it is this integrability that confines the mass of the approximating multipliers $z_i$, preventing its concentration at the nonregular instant. If $1/|\bar g_u|$ fails to be integrable near that instant, like in the example of Subsection \ref{subsec:example}, the mass of $z_i$ may concentrate at $t=0$ and only the degenerate charge of \eqref{trivial} survives, even when $\bar g_u(t)\neq 0$ for every $t\in(0,1]$.
	\end{remark}
	
	We can regain the existence of a normal tuple of multipliers by a different route: imposing the existence of a solution of the equation of variation associated with the dynamics of the problem, pointing into the admissible region on the set where the constraint is active.
	
	\begin{theorem}\label{t3}
		Assume \textbf{(C2)} and let $(\bar x,\bar u)$ be a weak local minimizer for (P) with $\bar g(0)=0$ and $\bar g_u(0)=0$. Suppose the following condition holds:
		\begin{itemize}
			\item[\textbf{(C3)}] there exists $(y,v)\in\AC{n}\times\Linf[m]$ solving the linear system
			\[\dot y(t)=\bar f_x(t)y(t)+\bar f_u(t)v(t),\qquad y(0)=0,\]
			and satisfying
			\[\bar g_x(t)y(t)+\bar g_u(t)v(t)<0\qquad\text{for a.e. }t\in\{\bar g=0\}.\]
		\end{itemize}
		Then there exists a tuple of multipliers $(\lambda,\beta,\lambda_0,z,\zeta)$ satisfying the conditions of Theorem \ref{pmp} with $\zeta=0$, $\beta\equiv 0$ and $\lambda_0>0$; in particular, after normalization one may take $\lambda_0=1$.
	\end{theorem}
	
	\subsection{Nondegeneracy and normality without regular approximations}
	
	Condition \textbf{(C3)} above, which together with \textbf{(C2)} yields normality under the approximability hypothesis \textbf{(R)}, can be strengthened so that \textbf{(C2)}, \textbf{(R)}, and approximations altogether are no longer needed, yielding the stronger conclusion that \emph{every} tuple of multipliers is normal. This strengthening, however, is not applicable when $t^*=0$ is a nonregular instant, so the result below should be viewed as complementary to Theorems \ref{t1}, \ref{t2}, and \ref{t3} rather than as subsuming them. For an in-depth discussion of the scope of Theorem \ref{t4} and how it compares to the earlier results, see Remarks \ref{rem:t4scope1} and \ref{rem:t4scope2}.
	
	\begin{theorem}\label{t4}
		Let $(\bar x,\bar u)$ be a weak local minimizer for (P) and suppose the following condition holds:
		\begin{itemize}
			\item[\textbf{(C4)}] there exist $\varepsilon>0$, $\delta>0$ and $(y,v)\in \AC{n}\times \Linf[m]$ solving the linear system
			\[\dot{y}(t)=\bar{f}_x(t)y(t)+\bar{f}_u(t)v(t),\quad y(0)=0,\]
			and satisfying
			\[\bar{g}_x(t)y(t)+\bar{g}_u(t)v(t)\leq-\delta\quad\mbox{for a.e. }t\in \{\bar{g}\geq-\varepsilon\}.\]
		\end{itemize}
		Then \emph{every} tuple of multipliers $(\lambda,\beta,\lambda_0,z,\zeta)$ satisfying the conditions of Theorem \ref{pmp} fulfills $\lambda_0>0$ together with the bound
		\begin{equation}\label{massbound}
			\|z\|_{L^1}+\zeta([0,1])\;\leq\;\frac{\lambda_0}{\delta}\int_0^1\bigl|\bar L_x(t)y(t)+\bar L_u(t)v(t)\bigr|\,dt .
		\end{equation}
		In particular, the set of all multipliers satisfying $\lambda_0=1$ is bounded.
	\end{theorem}
	
	\begin{remark}\label{rem:t4scope1}
		The standing assumptions $\bar g(0)=0$ and $\bar g_u(0)=0$ of the previous results were deliberately dropped from Theorem \ref{t4}: under them, condition \textbf{(C4)} is unsatisfiable. Indeed, $\bar g(0)=0$ forces $A_\varepsilon$ to contain almost all of a neighborhood of the origin, while $y(0)=0$, the essential boundedness of $\bar g_x$ and the essential vanishing of $\bar g_u$ at $t=0$ give $\bar g_x y+\bar g_u v\to 0$ essentially as $t\to 0^+$ for every pair $(y,v)$ admissible in \textbf{(C4)}, so no uniform margin can hold near the origin. This is exactly as it must be: in that scenario Proposition \ref{prop:trivial}(i) exhibits a degenerate tuple satisfying all conditions of Theorem \ref{pmp}, so no condition on $(\bar x,\bar u)$ can yield the strong conclusion of Theorem \ref{t4}. The same obstruction arises at a nonregular instant $t^*\in(0,1]$ with $\bar g_x(t^*)=0$, in accordance with Proposition \ref{prop:trivial}(ii).
	\end{remark}
	
	\begin{remark}\label{rem:t4scope2}
		When the nonregularities are located at instants $t^*\in(0,1]$ such that $\bar g_x(t^*)\neq 0$, then the strict margin in \textbf{(C4)} becomes attainable whenever the linearized system is suitably controllable: one may steer $y$ into the half-space $\bar g_x(t^*)y<0$, choosing for instance $v$ to vanish on the near-active set, so that there $\bar g_xy+\bar g_uv=\bar g_xy$. In particular this covers nonregularity at the terminal instant $t=1$ where, by \cite[Example 4.3]{be2021} (see Example \ref{ex:be43}), the pure charge cannot in general be nullified: Theorem \ref{t4} does not nullify it either, but ensures that every tuple of multipliers is normal and that the total mass of the charge obeys \eqref{massbound}. Theorem \ref{t4} thus complements, rather than supersedes, Theorems \ref{t1}--\ref{t3}: those results treat the nonregular initial instant by approximation with regular problems, while Theorem \ref{t4} treats nonregular instants elsewhere and dispenses with approximation altogether.
	\end{remark}

	\subsection{Normality with nonregularity at both endpoints}\label{subsec:mixed}
	
	Theorems \ref{t1}--\ref{t3} treat a nonregular \emph{initial} instant, whereas Theorem \ref{t4} treats nonregular instants elsewhere and dispenses with approximation, at the cost of the standing assumptions $\bar g(0)=0$, $\bar g_u(0)=0$ which, by Remark \ref{rem:t4scope1}, render the margin condition in \textbf{(C4)} unsatisfiable at the origin. Neither framework covers, on its own, a process that is nonregular \emph{simultaneously} at the initial instant of the type governed by Theorems \ref{t1}--\ref{t3}, and at the terminal instant beyond the reach of those results. The present subsection provides the missing bridge. Before stating the result, let us formulate the assumptions required: 
	
	\begin{itemize}
		\item[\textbf{(R-mid)}] For every $0<a<b<1$ the restriction of $(\bar x,\bar u)$ to $[a,b]$ is regular in the sense of Definition \ref{reg}.
		\item[\textbf{(C2-loc)}] there exist $\varepsilon_0\in(0,1)$, $\varepsilon>0$ and a non-decreasing $\omega:(0,\varepsilon_0]\to(0,\infty)$ with 
		\[
		\int_0^{\varepsilon_0}\omega(t)^{-1}\,dt<\infty\qquad\text{and}\qquad |\bar g_u(t)|\ge\omega(t) \text{ a.e. on } A_\varepsilon\cap[0,\varepsilon_0]
		\]
		\item[\textbf{(C4$\setminus$0)}] there exist $\varepsilon>0$ and $(y,v)\in\AC{n}\times\Linf[m]$ solving the linear system
		\[\dot y(t)=\bar f_x(t)y(t)+\bar f_u(t)v(t),\qquad y(0)=0.\]
		such that, with $w:=\bar g_xy+\bar g_uv$, one has $w\le 0$ a.e. on $A_\varepsilon$, and
		\[\text{for every }\eta\in(0,1)\text{ there is }\delta_\eta>0\text{ with }w(t)\le-\delta_\eta\ \text{ a.e. on }A_\varepsilon\cap[\eta,1].\]
	\end{itemize}
	
	\begin{remark}
		Notice that these conditions are modified versions to the ones used previously: \textbf{(R-mid)} is the regularity of each interval not containing the endpoints; \textbf{(C2-loc)} is just \textbf{(C2)} localized around the origin; and \textbf{(C4$\setminus$0)} is \textbf{(C4)} with the uniform margin allowed to die at the origin (where Remark \ref{rem:t4scope1} shows no such margin can survive) and imposing a weaker nonstrict inequality $w\le 0$ on the whole near-active set, but strengthening it to a strict $\eta$-dependent margin away from the origin. As $\eta\to0^+$, the margin $\delta_\eta$ in \textbf{(C4$\setminus$0)} degenerates to zero, on account of the equalities $\bar g(0)=\bar g_u(0)=0$.
	\end{remark}
	
	We can now state the result.
	
	\begin{theorem}\label{t5}
		Let $(\bar x,\bar u)$ be a weak local minimizer for (P) with $\bar g(0)=0$ and $\bar g_u(0)=0$, and assume \textbf{(R-mid)},  \textbf{(C2-loc)}, and \textbf{(C4$\setminus$0)}. Then there exists a tuple of multipliers $(\lambda,\beta,\lambda_0,z,\zeta)$ satisfying the conditions of Theorem \ref{pmp} with $\lambda_0>0$ whose pure charge either vanishes or is left-concentrated at $t=1$. Moreover, for every $\eta\in(0,1)$,
		\begin{equation}\label{massboundmixed}
			\zeta([0,1])+\int_{[\eta,1]}z(t)\,dt\;\le\;\frac{\lambda_0}{\delta_\eta}\int_0^1\bigl|\bar L_x(t)y(t)+\bar L_u(t)v(t)\bigr|\,dt.
		\end{equation}
	\end{theorem}
	
	\begin{remark}\label{rem:t5scope_a}
		Theorem \ref{t5} interpolates between these two families of results. The initial instant is treated as in Theorems \ref{t2}--\ref{t3}, using approximation and the integrable bound \textbf{(C2-loc)}. In contrast, the terminal instant is handled as in Theorem \ref{t4} via the linearized margin \textbf{(C4$\setminus$0)}, which coincides with \textbf{(C4)} away from the origin. A major distinction arises, however, in the proof of this result: the normalized multiplier tuples for the auxiliary problems on $[t_i, 1]$ may carry nonvanishing charges $\zeta_i$ concentrated at $t=1$, while mass from the functions $z_i$ may also escape at $t=1$ due to the lack of an integrability hypothesis on $|\bar g_u|^{-1}$ near that point. Consequently, we resort to Banach limits to handle these sequences with limits in the dual space of $L^\infty$.
	\end{remark}
	
	\begin{remark}\label{rem:t5scope_b}
		In contrast with Theorem \ref{t4}, however, the conclusion is not that \emph{every} tuple is normal: the standing hypotheses $\bar g(0)=\bar g_u(0)=0$ keep Proposition \ref{prop:trivial}(i) in force, so a degenerate tuple, right-concentrated at $t=0$, always coexists with the normal tuple that Theorem \ref{t5} constructs. What the theorem secures is the \emph{existence} of a nondegenerate normal tuple whose charge, if any, sits only at the terminal instant, together with the quantitative control \eqref{massboundmixed}. Example \ref{ex:mixed} exhibits precisely this configuration.
	\end{remark}
	
	\begin{remark}\label{rem:t5variant}
		The initial-side hypothesis \textbf{(C2-loc)} may be replaced by the localized form of \textbf{(C1)}: an integrable $v_1$ with $\bar f_uv_1=0$ and $\langle\bar g_u,v_1\rangle\ge 1$ a.e. on $\{\bar g=0\}\cap[0,\varepsilon_0]$. Both hypotheses serve the single purpose of supplying a fixed integrable majorant for the approximating densities near $t=0$, and either suffices: under the \textbf{(C1)}-localized variant the majorant is $\|\bar L_u\|_{L^\infty}|v_1|$, as in the proof of Theorem \ref{t1}, and the pointwise bound \eqref{star} forces the density of the limit tuple to vanish a.e. on $[0,\varepsilon_0]$ once $\lambda_0=0$, so that the nonstrict part of \textbf{(C4$\setminus$0)} is not needed on $A_\varepsilon\cap[0,\eta]$ in that variant. We state Theorem \ref{t5} under \textbf{(C2-loc)} because it is the form in Example \ref{ex:mixed}: at $m=1$ the kernel direction $\bar f_uv_1=0$ with $v_1\neq0$ demanded by \textbf{(C1)} is unavailable when the control enters the dynamics.
	\end{remark}
	
	\begin{remark}\label{rem:t5balance}
		Suppose, in addition, that $|\bar g_u|^{-1}$ fails to be integrable at the terminal instant in the sense of \eqref{nonintterm} and that the data admit essential limits at $t=1$. Then the terminal balance identity of Proposition \ref{prop:termbalance} applies to the tuple constructed by Theorem \ref{t5}. Since that tuple is normal, branch (i) of the proposition is excluded; hence either $\bar L_u(1)=0$, or the data lie on the ray of branch (ii) and the terminal mass is pinned to $\zeta_1=c_0\lambda_0$, with $c_0>0$ determined by $\bar L_u(1)^T=-c_0\,\bar f_u(1)^T\bar g_x(1)^T$. No separate compatibility hypothesis is therefore needed: the mere existence of the normal tuple forces the data into the non-abnormal branch. This is the configuration realized with $c_0=2$ by Example \ref{ex:mixed}.
	\end{remark}
	
	\section{Examples}\label{sec:examples}
	
	This section illustrates the results of Sections \ref{sec:when} and \ref{sec:nondeg} on problems whose multipliers can be computed in closed form, so that each condition can be seen at work and its reach delimited. Subsection \ref{subsec:exC2C3} exhibits, on a one-parameter family with a scalar control, the integrability threshold behind conditions \textbf{(C2)}--\textbf{(C3)}; Subsection \ref{subsec:exterminal} transports that family to the terminal instant and completes an initial/terminal, integrable/non-integrable classification of the model problems of the paper; Subsection \ref{subsec:exC1} realizes condition \textbf{(C1)} with a two-dimensional control, its defining direction $\bar f_uv=0$ with $v\neq 0$ being unavailable to a scalar control that enters the dynamics; Subsection \ref{subsec:exC4} verifies condition \textbf{(C4)} on the terminal problem (T) of Example \ref{ex:be43}; Subsection \ref{subsec:exbranches} delimits, with two further families, the two branches of the terminal balance identity of Proposition \ref{prop:termbalance}; and Subsection \ref{subsec:exmixed} exhibits a two-parameter problem, nonregular at \emph{both} endpoints, realizing the mixed regime of Theorem \ref{t5}.
	
	\subsection{The initial instant: the integrability threshold of \textbf{(C2)} and normality via \textbf{(C3)}}\label{subsec:exC2C3}
	
	The one-parameter family introduced next illustrates Theorems \ref{t2} and \ref{t3}, and shows the integrability requirement in \textbf{(C2)} to be sufficient but not necessary. For each $\alpha\in(0,1]$ consider
	\begin{equation*}
		\text{(E$_\alpha$)}~\left\{	\begin{array}{lll}
			&\text{min}&J(x,u)=\displaystyle\int_0^1 u^2(t)\,dt,\\
			& s.t.&(x,u)\in \AC{}\times \Linf,\\
			&&\dot x(t)=u(t)\quad \text{a.e.},\\
			&&t+t^\alpha-x(t)-t^\alpha u(t)\leq 0\quad \text{a.e.},\\
			&&x(0)=0,
		\end{array}\right.
	\end{equation*}
	so that $f(t,x,u)=u$, $L(t,x,u)=u^2$ and $g(t,x,u)=t+t^\alpha-x-t^\alpha u$. The constraint is affine in $(x,u)$ and the cost is convex, hence (E$_\alpha$) is a convex problem; consequently any admissible process satisfying the conditions of Theorem \ref{pmp} with $\lambda_0\neq 0$ is a (global, hence weak local) minimizer.
	
	The process $(\bar x(t),\bar u(t))=(t,1)$ is admissible: $\dot{\bar x}=1=\bar u$, $\bar x(0)=0$, and $\bar g(t)=t+t^\alpha-t-t^\alpha=0$, so the constraint is active throughout $[0,1]$. Along this process
	\[\bar g(t)\equiv 0,\qquad \bar g_x(t)\equiv -1,\qquad \bar g_u(t)=-t^\alpha,\]
	all continuous on $[0,1]$, with $\bar g(0)=0$ and $\bar g_u(0)=0$: the process is nonregular precisely at the initial instant, where the degeneracy phenomenon of Section \ref{sec:when} may arise.
	
	\begin{example}[The nondegenerate regime $\alpha\in(0,1)$]\label{ex:nondeg}
		Let $\alpha\in(0,1)$. Since $\bar g\equiv 0$ we have $A_\varepsilon=[0,1]$, and taking the non-decreasing function $\omega(t)=t^\alpha$ we obtain $|\bar g_u(t)|=t^\alpha\geq\omega(t)$ together with $\int_0^1\omega(t)^{-1}dt=\int_0^1 t^{-\alpha}dt=\tfrac{1}{1-\alpha}<\infty$. Hence condition \textbf{(C2)} holds, and Theorem \ref{t2} guarantees nondegenerate multipliers with vanishing charge. Moreover, the pair $(y,v)=(t,1)$ solves $\dot y=\bar f_x y+\bar f_u v=v$, $y(0)=0$, and satisfies $\bar g_x(t)y(t)+\bar g_u(t)v(t)=-t-t^\alpha<0$ for $t\in(0,1]$; hence \textbf{(C3)} holds as well, and Theorem \ref{t3} ensures a multiplier with $\lambda_0\neq 0$.
		
		These multipliers can be exhibited in closed form. Set $\beta\equiv 0$ and $\zeta=0$. The stationarity condition reads $0=\lambda+2-t^\alpha z$, while the costate equation reduces to $\dot\lambda=z$; together with the transversality condition $\lambda(1)=0$ these give $\lambda(t)=t^\alpha z(t)-2$ and the linear equation $\dot z=\big(t^{-\alpha}-\alpha t^{-1}\big)z$, whose solution with $\lambda(1)=0$ (equivalently $z(1)=2$) is
		\[z(t)=2\,t^{-\alpha}\exp\!\left(\frac{t^{1-\alpha}-1}{1-\alpha}\right),\qquad \lambda(t)=2\exp\!\left(\frac{t^{1-\alpha}-1}{1-\alpha}\right)-2.\]
		One checks directly that $\lambda_0=1$, $z\geq 0$, $\beta\equiv 0$, $\zeta=0$ and $\lambda\in \operatorname{BV}$ satisfy all the conditions of Theorem \ref{pmp}, so $(\bar x,\bar u)$ is the minimizer of (E$_\alpha$).
		
		It is instructive to note that $z(t)\sim 2e^{-1/(1-\alpha)}\,t^{-\alpha}$ as $t\to 0^+$, so the costate multiplier $z$ blows up at the nonregular instant yet remains integrable precisely because $\alpha<1$. This is exactly the behaviour predicted by the proof of Theorem \ref{t2}, where the uniform bound $0\leq z_i\leq K/\omega = K\,t^{-\alpha}$ confines the mass of the approximating multipliers and prevents its concentration at $t=0$.
	\end{example}
	
	\begin{example}[The borderline regime $\alpha=1$]\label{ex:borderline}
		For $\alpha=1$ the constraint becomes $2t-x-tu\leq 0$ and $\bar g_u(t)=-t$, so that
		\[\int_0^1\frac{dt}{|\bar g_u(t)|}=\int_0^1\frac{dt}{t}=\infty.\]
		No non-decreasing $\omega$ with $\int_0^1\omega^{-1}<\infty$ can satisfy $t=|\bar g_u(t)|\geq\omega(t)$, so the integrability requirement of condition \textbf{(C2)} fails. Taking the limit $\alpha\to 1^-$ in the formulas above (or solving the corresponding equations directly) yields the constant multiplier $z\equiv 2$, $\lambda(t)=2t-2$, $\lambda_0=1$, $\beta\equiv 0$, $\zeta=0$, which still satisfies Theorem \ref{pmp}: the problem remains nondegenerate. This shows that condition \textbf{(C2)} is sufficient but not necessary; the favourable outcome here is owed to $\bar g_x\equiv -1$ being bounded away from zero and to the cost depending on the control alone.
	\end{example}
	
	\begin{remark}
		The contrast between Examples \ref{ex:nondeg} and \ref{ex:borderline} delimits the reach of condition \textbf{(C2)}. When the integrable lower bound on $|\bar g_u|$ fails, nondegeneracy is no longer guaranteed by Theorem \ref{t2}, and whether it is actually lost depends on further structure of the problem. Example \ref{ex:borderline} remains nondegenerate; but if, in addition, $\bar g_x$ vanishes or the cost is driven by the state, the borderline rate $\bar g_u\sim t$ no longer admits any integrable costate: as established in the case study of Subsection \ref{subsec:example}, the mass of the multipliers is then forced to concentrate at $t=0$, and only the degenerate charge of \eqref{trivial} survives, so the problem degenerates completely. Conditions \textbf{(C2)}--\textbf{(C3)} are precisely tailored to exclude this scenario while still allowing the Hölder-type degenerations $|\bar g_u(t)|\geq c\,t^\alpha$, $0\leq\alpha<1$, of Example \ref{ex:nondeg}.
	\end{remark}
	
	\subsection{The terminal instant under integrability: an optional charge}\label{subsec:exterminal}
	
	The two regimes just described sit at the \emph{initial} instant. Their terminal counterparts behave differently, and the next example completes the picture by exhibiting, in closed form, the terminal mirror of the integrable regime of Example \ref{ex:nondeg}. It is also the example invoked in Remark \ref{rem:termbalanceb} to show that the nonintegrability hypothesis \eqref{nonintterm} of Proposition \ref{prop:termbalance} cannot be dropped.
	
	\begin{example}[The terminal integrable regime]\label{ex:terminalint}
		Let $\alpha\in(0,1)$ and consider
		\begin{equation*}
			\text{(E$^1_\alpha$)}~\left\{	\begin{array}{lll}
				&\text{min}&J(x,u)=\displaystyle\int_0^1 u^2(t)\,dt,\\
				& s.t.&(x,u)\in \AC{}\times \Linf,\\
				&&\dot x(t)=u(t)\quad \text{a.e.},\\
				&&t+(1-t)^\alpha-x(t)-(1-t)^\alpha u(t)\leq 0\quad \text{a.e.},\\
				&&x(0)=0,
			\end{array}\right.
		\end{equation*}
		obtained from (E$_\alpha$) by transporting the weight $t^\alpha$ to the terminal instant. The process $(\bar x(t),\bar u(t))=(t,1)$ is admissible and the constraint is active throughout: along it
		\[\bar g(t)\equiv 0,\qquad \bar g_x(t)\equiv -1,\qquad \bar g_u(t)=-(1-t)^\alpha,\]
		so the process is nonregular precisely at $t=1$, where $\bar g_u$ vanishes; since $\alpha<1$, the reciprocal $|\bar g_u|^{-1}=(1-t)^{-\alpha}$ is integrable near $t=1$, and $\bar g_x(1)=-1\neq 0$, so neither case of Proposition \ref{prop:trivial} applies. As with (E$_\alpha$), the problem is convex, so any tuple with $\lambda_0\neq 0$ certifies global optimality.
		
		\emph{All the multiplier tuples can be computed.} On every subinterval $[0,1-\delta]$, $\delta\in(0,1)$, the process is regular. Take $v_0(t)=-(1-t)^{-\alpha}$, bounded there, so Lemma \ref{lem:chargesupport} shows that the pure charge $\zeta$ of any tuple satisfying Theorem \ref{pmp} vanishes on the measurable subsets of $[0,1-\delta]$ for every $\delta$: either $\zeta=0$, or $\zeta$ is left-concentrated at $t=1$ in the sense of Definition \ref{charges}(v). Write $m\coloneqq\zeta([0,1])\geq 0$. The defining formula in Theorem \ref{pmp}(c), with $\bar g_x\equiv-1$, gives $\beta(t)=\zeta\bigl((t,1]\bigr)=m$ for $t\in[0,1)$ and $\beta(1)=0$. Since $\bar f_x\equiv 0$, $\bar L_x\equiv 0$, $\bar f_u\equiv 1$ and $\bar L_u\equiv 2$, the first stationarity condition reads $\lambda(t)=(1-t)^\alpha z(t)-2\lambda_0$ a.e., while the costate equation states that $p\coloneqq\lambda+\beta$ is absolutely continuous with $\dot p=z$ and $p(1)=0$. Eliminating $\lambda$ yields the linear equation $(1-t)^\alpha\dot z=\bigl[1+\alpha(1-t)^{\alpha-1}\bigr]z$ on $[0,1)$, whose solutions compatible with $p(1)=0$ form the one-parameter family, indexed by $m\in[0,2\lambda_0]$ with $C\coloneqq 2\lambda_0-m\geq 0$,
		\begin{equation}\label{Etermfamily}
			z(t)=C\,(1-t)^{-\alpha}\exp\!\left(-\frac{(1-t)^{1-\alpha}}{1-\alpha}\right),\qquad
			\lambda(t)=C\,\exp\!\left(-\frac{(1-t)^{1-\alpha}}{1-\alpha}\right)-2\lambda_0\quad\text{on }[0,1),
		\end{equation}
		with $\lambda(1)=0$; one has $z\geq 0$ if and only if $C\geq 0$, that is, $m\leq 2\lambda_0$, and $z\in\Lone$ for every such $C$ because $\alpha<1$. Three findings deserve emphasis.
		
		First, a \emph{charge-free normal tuple exists}: taking $m=0$ and $\lambda_0=1$ in \eqref{Etermfamily} gives $\zeta=0$, $\beta\equiv 0$ and an absolutely continuous costate; the density blows up like $(1-t)^{-\alpha}$ at the nonregular instant yet remains integrable, in exact mirror image of Example \ref{ex:nondeg}. Integrable terminal nonregularity therefore does \emph{not} force a charge.
		
		Second, \emph{charged normal tuples coexist} for every $m\in(0,2\lambda_0]$. At the extreme $m=2\lambda_0$, that is $C=0$, the density vanishes identically, $\lambda\equiv-2\lambda_0$ on $[0,1)$, and the tuple carries a pure charge of mass $2\lambda_0$ left-concentrated at $t=1$; it has the structure of the tuple exhibited for (T) in Remark \ref{rem:be43}. The charge is here optional: neither forced, as in Example \ref{ex:be43}, nor forbidden.
		
		Third, \emph{every tuple is normal}: if $\lambda_0=0$, then $C=-m\leq 0$, and $z\geq 0$ forces $C=0$; hence $m=0$ and $z\equiv 0$, contradicting the nontriviality condition. In particular no degenerate tuple exists and, by the first finding, the problem is as well behaved as the regular theory would predict, despite the terminal nonregularity.
	\end{example}
	
	\begin{remark}[The initial/terminal, integrable/non-integrable asymmetry]\label{rem:asym}
		Example \ref{ex:terminalint} fills the last cell of a two-by-two picture of the model problems of the paper, organized by the location of the isolated nonregular instant and by the integrability of $|\bar g_u|^{-1}$ near it:
		\begin{center}
			\begin{tabular}{l|l|l}
				& \emph{initial instant} $t^*=0$ & \emph{terminal instant} $t^*=1$\\
				\hline
				$|\bar g_u|^{-1}$ non-integrable & (E): completely degenerate; only a & (T): pure charge forced, yet every\\
				& pure charge survives (Subsection \ref{subsec:example}) & tuple is normal (Remark \ref{rem:be43})\\
				\hline
				$|\bar g_u|^{-1}$ integrable & (E$_\alpha$): normal charge-free tuple & (E$^1_\alpha$) $\alpha\in(0,1)$: every tuple normal; charge\\
				& (Example \ref{ex:nondeg}); the degenerate tuple & \emph{optional}: charge-free and charged\\
				& of Proposition \ref{prop:trivial}(i) coexists & tuples coexist (Example \ref{ex:terminalint})\\
			\end{tabular}
		\end{center}
		In (E), (E$_\alpha$) and (E$^1_\alpha$) the constraint is active throughout $[0,1]$, whereas in (T) it is active only at $t=1$. Two comments are in order. (a) The columns are not symmetric. At the initial instant, Proposition \ref{prop:trivial}(i) always furnishes a degenerate tuple (cf. Remark \ref{rem:t4scope1}); whether informative multipliers coexist with it depends on finer structure; compare (E), completely degenerate, with the borderline problem of Example \ref{ex:borderline}, equally non-integrable at $t=0$ yet nondegenerate. In the two terminal problems, where $\bar g_x(1)\neq 0$, no degenerate tuple exists at all, and the question becomes whether the pure charge is forced: it is in (T), and it is optional in (E$^1_\alpha$). The examples might suggest that both the integrability of $|\bar g_u|^{-1}$ and the local activation pattern of the constraint intervene in this alternative; Proposition \ref{prop:termbalance} shows that, whenever the data admit essential limits at $t=1$, the activation pattern is in fact immaterial, and the alternative is decided by the integrability of $|\bar g_u|^{-1}$ together with the position of $\bar L_u(1)^T$ relative to the ray spanned by $-\bar f_u(1)^T\bar g_x(1)^T$. (b) Condition \textbf{(C1)} interacts with the rows: by the Cauchy--Schwarz inequality, \textbf{(C1)} gives $1\leq\langle\bar g_u,v\rangle\leq|\bar g_u|\,|v|$, hence $|\bar g_u|^{-1}\leq|v|\in\Lone$, a.e. on the active set $\{\bar g=0\}$. Consequently, whenever \textbf{(C1)} holds, a nonregular instant around which the constraint is active on a set of full measure, as at $t=1$ in (E$^1_\alpha$) with $\alpha\in(0,1)$, necessarily falls in the integrable row, where the terminal example shows a charge-free tuple to be available. \textbf{(C1)} is silent, by contrast, on instants with thin activation, such as the terminal instant of (T), where the active set is Lebesgue null, \textbf{(C1)} holds vacuously, and the charge is nonetheless forced.
	\end{remark}
	
	\subsection{Condition \textbf{(C1)} and Theorem \ref{t1}: the role of the kernel of $\bar f_u$}\label{subsec:exC1}
	
	\begin{example}[A two-control problem illustrating \textbf{(C1)}]\label{ex:C1}
		Fix $\alpha\in(0,1)$ and consider the problem with scalar state and
		two-dimensional control
		\begin{equation*}
			\text{(M$_\alpha$)}~\left\{\begin{array}{lll}
				&\text{min}& J(x,u)=\displaystyle\int_0^1\Big(\tfrac12 u_1(t)^2+u_2(t)\Big)\,dt,\\
				& s.t.&(x,u)\in \AC{}\times \Linf[2],\\
				&&\dot x(t)=u_1(t)\quad\text{a.e.},\\
				&&t^\alpha\bigl(1-u_2(t)\bigr)\leq 0\quad\text{a.e.},\\
				&&x(0)=0,
			\end{array}\right.
		\end{equation*}
		so that $f(t,x,u)=u_1$, $L(t,x,u)=\tfrac12u_1^2+u_2$ and
		$g(t,x,u)=t^\alpha(1-u_2)$, with $g$ independent of $x$.
		
		The problem decouples: the mixed constraint reads $u_2\geq 1$ a.e., so
		minimizing $\int_0^1 u_2$ forces $\bar u_2\equiv 1$, while minimizing
		$\tfrac12\int_0^1 u_1^2$ with free right endpoint forces $\bar u_1\equiv 0$
		and $\bar x\equiv 0$. Hence $(\bar x,\bar u)=(0,(0,1))$ is the unique global
		minimizer. Along it
		\[\bar g\equiv 0,\qquad \bar g_x\equiv 0,\qquad \bar g_u=(0,-t^\alpha),\qquad
		\bar f_u=(1,0),\qquad \bar L_u=(0,1),\]
		so the constraint is active on all of $[0,1]$ and the process is nonregular
		precisely at $t=0$, where $\bar g_u(0)=0$; on each $[\delta,1]$ the bounded
		direction $v_0=(0,-\delta^{-\alpha})$ gives $\bar g_u v_0\geq 1$, so
		Definition \ref{reg} holds there and hypothesis \textbf{(R)} is satisfied.
		
		Condition \textbf{(C1)} holds with the \emph{integrable} direction
		$v=(0,-t^{-\alpha})\in\Lone[2]$: indeed $\bar f_u v=0$ and
		$\langle\bar g_u,v\rangle=(-t^\alpha)(-t^{-\alpha})=1$ on the active set $\{\bar g=0\}=[0,1]$,
		while $\|v\|_{L^1}=\int_0^1 t^{-\alpha}\,dt=\tfrac{1}{1-\alpha}<\infty$ because
		$\alpha<1$. Note that $v\notin\Linf[2]$; the bounded directions of
		Definition \ref{reg} do not suffice, and it is essential that \textbf{(C1)}
		admits $L^1$ directions.
		
		Theorem \ref{t1} therefore applies, and the multipliers can be exhibited in
		closed form. Since $\bar g_x\equiv 0$ we have $\beta\equiv 0$ and $\lambda$
		absolutely continuous; the costate equation $-\dot\lambda=0$ with
		$\lambda(1)=0$ gives $\lambda\equiv 0$, and the first stationarity condition
		$\lambda_0-t^\alpha z=0$ gives $z(t)=\lambda_0 t^{-\alpha}$, which lies in
		$\Lone$ exactly because $\alpha<1$. Taking $\zeta=0$ and normalizing
		$\lambda_0=1$,
		\[\lambda_0=1,\qquad \lambda\equiv 0,\qquad \beta\equiv 0,\qquad
		z(t)=t^{-\alpha},\qquad \zeta=0,\]
		satisfies all conditions of Theorem \ref{pmp}: the multipliers are normal,
		charge-free, and the costate is absolutely continuous, as Theorem \ref{t1}
		predicts. Moreover the lower bound of Theorem \ref{t1} is sharp here: with
		$\|\bar L_u\|_{L^\infty}=1$ and $\|v\|_{L^1}=\tfrac{1}{1-\alpha}$ it reads
		$\lambda_0\geq\tfrac{1-\alpha}{2-\alpha}$, and this value is attained under the
		normalization $\lambda_0+\|z\|_{L^1}=1$.
		
		Finally, $(M_\alpha)$ shares with Example \ref{ex:E} the structure
		$\bar g_x\equiv 0$ and a constraint active on a full neighborhood of the
		nonregular instant $t=0$, which rendered the \emph{scalar} problem (E)
		completely degenerate. Here the second control supplies an inward variation
		$v\in\ker\bar f_u$ that relaxes the constraint without disturbing the
		state and \textbf{(C1)} converts it into a normal multiplier. The problem
		still admits the degenerate tuple of Proposition \ref{prop:trivial}(i), since
		$t^*=0$; Theorem \ref{t1} does not remove it, but furnishes a nondegenerate
		normal tuple alongside it, so that, unlike (E), the problem is not completely
		degenerate.
	\end{example}
	
	\subsection{Condition \textbf{(C4)} and Theorem \ref{t4}: normality of every tuple}\label{subsec:exC4}
	
	\begin{example}[Condition \textbf{(C4)} at a terminal instant]\label{ex:C4}
		We revisit the terminal problem (T) of Example \ref{ex:be43} and verify that
		it satisfies condition \textbf{(C4)}, so that Theorem \ref{t4} applies. Along
		the minimizer $(\bar x,\bar u)=(t,0)$ one has $\bar f_x\equiv 0$,
		$\bar f_u\equiv -1$, $\bar g_x\equiv 1$, $\bar g_u\equiv 0$, $\bar L_x\equiv 0$
		and $\bar L_u\equiv 1$, and the constraint is active only at $t=1$, so
		$A_\varepsilon=\{\bar g\geq-\varepsilon\}=[1-\varepsilon,1]$. Take
		$v\equiv 1$ and let $y$ solve $\dot y=\bar f_x y+\bar f_u v=-1$, $y(0)=0$, that
		is $y(t)=-t$. Then
		\[\bar g_x(t)y(t)+\bar g_u(t)v(t)=y(t)=-t\leq-(1-\varepsilon)\qquad\text{on }A_\varepsilon,\]
		so \textbf{(C4)} holds with $\delta=1-\varepsilon$ for every
		$\varepsilon\in(0,1)$. By Theorem \ref{t4}, \emph{every} tuple of multipliers
		for (T) is normal, $\lambda_0>0$, and obeys the mass bound \eqref{massbound},
		which here reads
		\[\|z\|_{L^1}+\zeta([0,1])\ \leq\ \frac{\lambda_0}{1-\varepsilon}\int_0^1\bigl|\bar L_x(t)y(t)+\bar L_u(t)v(t)\bigr|\,dt=\frac{\lambda_0}{1-\varepsilon},\]
		since $\bar L_x y+\bar L_u v\equiv 1$. Optimizing over $\varepsilon\in(0,1)$
		gives $\|z\|_{L^1}+\zeta([0,1])\leq\lambda_0$, and this is sharp: the normal
		tuple of Remark \ref{rem:be43}, with $z\equiv 0$ and a charge $\zeta$ of mass
		$\lambda_0$ left-concentrated at $t=1$, attains it with equality. Thus,
		although the pure charge cannot be nullified at this terminal nonregular
		instant (Example \ref{ex:be43}), Theorem \ref{t4} keeps the maximum principle
		fully informative: no tuple is degenerate and the mass of the charge is
		controlled.
	\end{example}
	
	\subsection{The two branches of the terminal balance identity of Proposition \ref{prop:termbalance}}\label{subsec:exbranches}
	
	The terminal column of Remark \ref{rem:asym} exhibits two regimes: at the integrable instant of Example \ref{ex:terminalint} the pure charge is optional, its mass $m$ ranging freely over $[0,2\lambda_0]$, whereas in the non-integrable problem (T) the charge is forced (Example \ref{ex:be43}) and the sharp bound of Example \ref{ex:C4} is attained by a charge of mass exactly $\lambda_0$. Both findings are accounted for by the terminal balance identity of Proposition \ref{prop:termbalance}: at a terminal instant where $|\bar g_u|^{-1}$ fails to be integrable, the stationarity condition, the integrability of the density $z$ and the transversality condition combine into an exact balance between the cost multiplier and the mass that the charge deposits at $t=1$. Recall that its hypotheses bear on the problem data along the nominal process alone; no assumption is made on the multipliers, nor on the activation pattern of the constraint. The two branches of that proposition are both realized, and both under a constraint active on all of $[0,1]$; the extreme opposite of the ``thin'' activation of (T).

	\begin{example}[The terminal non-integrable regime under full activation]\label{ex:terminalforced}
		Let $\alpha\geq 1$ and consider the problem (E$^1_\alpha$) defined by the formulas of Example \ref{ex:terminalint}. This range for $\alpha$ now places the family of problems in the non-integrable cells of Remark \ref{rem:asym}. The process $(\bar x,\bar u)=(t,1)$ remains admissible, with
		\[\bar g\equiv 0,\qquad\bar g_x\equiv-1,\qquad\bar g_u(t)=-(1-t)^\alpha,\qquad\bar f_u\equiv 1,\qquad\bar L_u\equiv 2,\]
		so the constraint is active on all of $[0,1]$ and the process is nonregular precisely at $t=1$; but now $|\bar g_u|^{-1}=(1-t)^{-\alpha}$ fails to be integrable near $1$, so hypothesis \eqref{nonintterm} of Proposition \ref{prop:termbalance} holds, and the data, being continuous, admit essential limits everywhere. Since $\bar L_u(1)=2$ and $\bar f_u(1)^T\bar g_x(1)^T=-1$, case (ii) of the proposition applies with $c_0=2$: \emph{every} tuple must satisfy $\zeta_1=2\lambda_0$.
		
		The tuples can again be computed in full, and confirm the prediction. As in Example \ref{ex:terminalint}, Lemma \ref{lem:chargesupport} confines the pure charge to the terminal instant, and the elimination performed there produces the same linear equation $(1-t)^\alpha\dot z=\bigl[1+\alpha(1-t)^{\alpha-1}\bigr]z$ on $[0,1)$, whose nonnegative solutions are now
		\[z(t)=C\,(1-t)^{-\alpha}\exp\!\left(\frac{(1-t)^{1-\alpha}-1}{\alpha-1}\right)\ \ (\alpha>1),\qquad z(t)=C\,(1-t)^{-2}\ \ (\alpha=1),\qquad C\geq 0.\]
		For $C>0$ these fail to be integrable near $t=1$ (for $\alpha>1$ the exponent even diverges to $+\infty$), so $z\in\Lone$ forces $C=0$. Hence $z\equiv 0$ and $\lambda=-2\lambda_0$ a.e. on $[0,1)$, and the absolute continuity of $p=\lambda+\beta$ together with $p(1)=0$ forces $\beta=2\lambda_0$ on $[0,1)$, that is, $\zeta\bigl((t,1]\bigr)=2\lambda_0$ for every $t<1$. If $\lambda_0=0$ the tuple vanishes identically, contradicting nontriviality; therefore \emph{every tuple is normal} and, after the normalization $\lambda_0=1$, consists of $z\equiv 0$, $\lambda=-2\chi_{[0,1)}$, $\beta=2\chi_{[0,1)}$, and a pure charge of mass $2$ left-concentrated at $t=1$ (such charges exist; see Remark \ref{rem:be43}). One verifies directly, as in Example \ref{ex:terminalint}, that this tuple fulfills all the conditions of Theorem \ref{pmp}, and the convexity of the problem, which is unaffected by $\alpha$, certifies that $(\bar x,\bar u)$ is a global minimizer for every $\alpha\geq 1$.
		
		Two situations deserve emphasis. First, the outcome is qualitatively that of (T); charge forced, mass pinned to the cost multiplier, and every tuple normal, although the activation pattern is the extreme opposite: no hypothesis phrased in terms of the activation pattern alone can force abnormality at a non-integrable terminal instant. Second, the contrast with the initial instant is complete: the same non-integrable rate placed at $t=0$, with $\bar g_x\equiv 0$, produces the total collapse of Subsection \ref{subsec:example}, whereas here the costate absorbs the escaping mass as a terminal jump whose size the balance identity prescribes.
	\end{example}
	
	\begin{example}[Complete degeneration at the terminal instant]\label{ex:terminaldeg}
		Case (i) of Proposition \ref{prop:termbalance} also occurs, and in the strongest possible form. For $\alpha\geq 1$, consider the terminal mirror of the two-control problem of Example \ref{ex:C1}:
		\begin{equation*}
			\text{(M$^1_\alpha$)}~\left\{\begin{array}{lll}
				&\text{min}& J(x,u)=\displaystyle\int_0^1\Big(\tfrac12 u_1(t)^2+u_2(t)\Big)\,dt,\\
				& s.t.&(x,u)\in \AC{}\times \Linf[2],\\
				&&\dot x(t)=u_1(t)\quad\text{a.e.},\\
				&&(1-t)^\alpha\bigl(1-u_2(t)\bigr)\leq 0\quad\text{a.e.},\\
				&&x(0)=0.
			\end{array}\right.
		\end{equation*}
		The constraint amounts to $u_2(t)\geq 1$ a.e., so the problem decouples exactly as in Example \ref{ex:C1} and $(\bar x,\bar u)=(0,(0,1))$ is the unique global minimizer. Along it,
		\[\bar g\equiv 0,\qquad\bar g_x\equiv 0,\qquad\bar g_u(t)=\bigl(0,-(1-t)^\alpha\bigr),\qquad\bar f_u\equiv(1,0),\qquad\bar L_u\equiv(0,1):\]
		the constraint is active throughout, the process is nonregular precisely at $t=1$, and \eqref{nonintterm} holds for $\alpha\geq 1$. Here $\bar f_u(1)^T\bar g_x(1)^T=0$ while $\bar L_u(1)=(0,1)\neq 0$: case (i) of Proposition \ref{prop:termbalance} ensures $\lambda_0=0$ for \emph{every} tuple.
		
		Direct computation confirms this, and more. Since $\bar g_x\equiv 0$, the formula in Theorem \ref{pmp}(c) gives $\beta\equiv 0$ and $\lambda$ absolutely continuous; the costate equation $-\dot\lambda=0$ with $\lambda(1)=0$ yields $\lambda\equiv 0$, and the first stationarity condition reduces to its second component, $\lambda_0=(1-t)^\alpha z(t)$ a.e. For $\lambda_0>0$ this would force $z(t)=\lambda_0(1-t)^{-\alpha}\notin\Lone$; hence $\lambda_0=0$, $z\equiv 0$, and the nontriviality condition demands $\zeta([0,1])>0$. On every interval $[0,1-\delta]$ the process is regular with $v_0=(0,-\delta^{-\alpha})$, so Lemma \ref{lem:chargesupport} forces $\zeta$ to be left-concentrated at $t=1$; conversely, every pure charge $0\neq\zeta\geq 0$ in $\ba$ left-concentrated at $t=1$ satisfies the remaining conditions of Theorem \ref{pmp}: the charge stationarity holds because $\bigl|\int_E(1-t)^\alpha\,d\zeta\bigr|\leq\varepsilon^\alpha\,\zeta([0,1])$ for every $\varepsilon>0$ and every $E\in\mathfrak{L}$, and the slackness conditions are trivial since $\bar g\equiv 0$. The tuples are therefore exactly the degenerate tuples of Definition \ref{def:trivial} with $t^*=1$ and $\lambda\equiv\beta\equiv 0$: \emph{the problem is completely degenerate}, now at the terminal instant. This realizes, for all tuples at once, the configuration of Proposition \ref{prop:trivial}(ii), and accords with Remark \ref{rem:t4scope1}: since $\bar g_x(1)=0$, no margin of the type \textbf{(C4)} can hold near $t=1$.
		
		The mechanism behind the two branches is the following. With $\bar g_x\equiv 0$ the terminal charge is invisible to the costate, so nothing can offset $\lambda_0\,\bar L_u(1)^T$ in \eqref{balance} and the cost multiplier must vanish. When instead $\bar f_u(1)^T\bar g_x(1)^T$ opposes $\bar L_u(1)^T$, the terminal jump of the costate balances the cost term, and normality survives with the pinned mass $\zeta_1=c_0\lambda_0$. Within the non-integrable terminal cell of Remark \ref{rem:asym}, then, (T) and (E$^1_\alpha$) with $\alpha\geq 1$ realize case (ii) while (M$^1_\alpha$) realizes case (i): the cell splits along the data ray of Proposition \ref{prop:termbalance}, not along the activation pattern. Finally, for $\alpha\in(0,1)$ the mirrored computation of Example \ref{ex:C1} furnishes the normal charge-free tuple $\lambda_0=1$, $\lambda\equiv 0$, $z(t)=(1-t)^{-\alpha}$, $\beta\equiv 0$, $\zeta=0$: the family (M$^1_\alpha$) crosses from fully informative to completely degenerate exactly at the integrability threshold, in mirror image of the initial-instant families of this section.
	\end{example}
	
	\subsection{Nonregularity at both endpoints: the mixed regime of Theorem \ref{t5}}\label{subsec:exmixed}
	
	\begin{example}[Nonregularity at both endpoints]\label{ex:mixed}
		Fix $\alpha\in(0,1)$ and $\gamma\ge 1$, write $\varphi(t)=t^\alpha(1-t)^\gamma$, and consider
		\begin{equation*}
			\text{(G$_{\alpha,\gamma}$)}~\left\{\begin{array}{lll}
				&\text{min}&J(x,u)=\displaystyle\int_0^1 u^2(t)\,dt,\\
				& s.t.&(x,u)\in \AC{}\times \Linf,\\
				&&\dot x(t)=u(t)\quad\text{a.e.},\\
				&&t+\varphi(t)-x(t)-\varphi(t)u(t)\le 0\quad\text{a.e.},\\
				&&x(0)=0.
			\end{array}\right.
		\end{equation*}
		The process $(\bar x,\bar u)=(t,1)$ is admissible, and the constraint is active throughout: along it
		\[\bar g\equiv 0,\qquad\bar g_x\equiv-1,\qquad\bar g_u=-\varphi,\qquad\bar f_u\equiv 1,\qquad\bar L_u\equiv 2,\qquad\bar f_x=\bar L_x\equiv 0.\]
		Since $\varphi$ vanishes only at $t=0$ and $t=1$, the process is nonregular exactly at the two endpoints, and $|\bar g_u|^{-1}=\varphi^{-1}$ is integrable at $0$ (because $\alpha<1$) but not at $1$ (because $\gamma\ge1$).
		
		For any admissible $(x,u)$ one has $\int_0^1u^2=\int_0^1(u-1)^2+2x(1)-1$, since $\int_0^1(2u-1)\,dt=2x(1)-1$. The constraint gives $x(t)\ge t+\varphi(t)\bigl(1-u(t)\bigr)$ a.e.; as $\varphi(t)\to0$ when $t\to1^-$ and $x$ is continuous, this forces $x(1)\ge 1$. Hence $J\ge 1=J(\bar x,\bar u)$, with equality if and only if $u\equiv1$: $(\bar x,\bar u)$ is the unique global minimizer.
		
		On every $[a,b]\subset(0,1)$ the bounded direction $v_0=-\varphi^{-1}$ gives $\bar g_uv_0\equiv1$, so \textbf{(R-mid)} holds. For \textbf{(C2-loc)}, take $\varepsilon_0=\tfrac12$ and $\omega(t)=2^{-\gamma}t^\alpha$: since $(1-t)^\gamma\ge2^{-\gamma}$ on $[0,\tfrac12]$, one has $|\bar g_u(t)|=t^\alpha(1-t)^\gamma\ge\omega(t)$ there, $\omega$ is non-decreasing, and $\int_0^{1/2}\omega^{-1}<\infty$ because $\alpha<1$. For \textbf{(C4$\setminus$0)}, take $(y,v)=(t,1)$, which solves $\dot y=v$, $y(0)=0$; then
		\[w=\bar g_xy+\bar g_uv=-t-\varphi(t)\le0\ \text{ on }[0,1],\qquad w\le-\eta\ \text{ on }[\eta,1],\]
		so $\delta_\eta=\eta$. Theorem \ref{t5} therefore applies, and (G$_{\alpha,\gamma}$) admits a normal tuple whose charge is confined to $t=1$.
		
		Condition \textbf{(C1)} is unavailable at $m=1$, since $\bar f_u\equiv1$ forces $v\equiv0$ in $\bar f_uv=0$, hence $\langle\bar g_u,v\rangle\equiv0<1$. The global condition \textbf{(C2)} fails, as a non-decreasing $\omega$ with integrable reciprocal cannot lie below $|\bar g_u|=\varphi\to0$ at $t=1$; and \textbf{(C3)}, which presupposes \textbf{(C2)}, fails with it. Finally \textbf{(C4)} fails near the origin exactly as in Remark \ref{rem:t4scope1}, because $\bar g(0)=\bar g_u(0)=0$. Only the mixed-regime Theorem \ref{t5} covers this process.
		
		The elimination proceeds as in Example \ref{ex:terminalint}. With $\bar g_x\equiv-1$ the charge deposits $\beta(t)=\zeta\bigl((t,1]\bigr)$; setting $p=\lambda+\beta$, the stationarity condition gives $\lambda=\varphi z-2\lambda_0$ while $\dot p=z$ and $p(1)=0$, hence $\varphi\dot z=(1-\varphi')z$ on $[0,1)$, with solutions
		\[z(t)=C\,\varphi(t)^{-1}\exp\!\left(\int_{1/2}^t\frac{ds}{\varphi(s)}\right),\qquad C\ge0.\]
		As $t\to1^-$ the exponent diverges to $+\infty$, like $(1-t)^{1-\gamma}/(\gamma-1)$ for $\gamma>1$ and logarithmically for $\gamma=1$, where $z\sim C(1-t)^{-2}$, so $z\notin\Lone$ unless $C=0$; at $t=0$ the exponent has a finite limit (as $\alpha<1$) and imposes nothing. Hence $C=0$: for \emph{every} tuple, $z\equiv0$, $\lambda\equiv-2\lambda_0$ on $(0,1)$, and $\zeta\bigl((t,1]\bigr)=2\lambda_0$ for every $t<1$, a pure charge with terminal mass at $t=1$ equal to $2\lambda_0$. This confirms Proposition \ref{prop:termbalance}, branch (ii) with $c_0=2$, since $\bar L_u(1)=2$ and $\bar f_u(1)^T\bar g_x(1)^T=-1$, in agreement with Remark \ref{rem:t5balance}.
		
		Taking $\lambda_0=1$ yields the normal tuple predicted by Theorem \ref{t5}: $z\equiv0$, $\lambda=-2\chi_{[0,1)}$, $\beta=2\chi_{[0,1)}$, and a charge of mass $2$ left-concentrated at $t=1$. But $\lambda_0=0$ does not empty the multiplier set: because $\bar g(0)=\bar g_u(0)=0$, Proposition \ref{prop:trivial}(i) furnishes a degenerate tuple with $\lambda_0=0$, $z\equiv0$, $\lambda\equiv0$, and a pure charge right-concentrated at $t=0$ (the charge stationarity $\int_E\bar g_u\,d\zeta=0$ holds because $\varphi\to0$ at the origin). The two coexist, exactly as Remark \ref{rem:t5scope_b} describes: unlike the terminal problems (T) and (E$^1_\alpha$) of Examples \ref{ex:be43} and \ref{ex:C4} and $(E_\alpha^1)$ of Examples \ref{ex:terminalint} and \ref{ex:terminalforced}, where the initial instant is regular and \emph{every} tuple is normal, here the initial nonregularity keeps a degenerate tuple alive alongside the normal one. Theorem \ref{t5} is sharp in claiming existence rather than universality.
	\end{example}
	
	\section{Proofs}\label{sec:proofs}
	
	This section collects the proofs of Theorems \ref{t1}--\ref{t5}. Those of Theorems \ref{t1} and \ref{t2} (Subsections \ref{subsec:prooft1} and \ref{subsec:prooft2}) share a single scheme: multipliers are produced for the regular truncations of the problem, conditions \textbf{(C1)} and \textbf{(C2)} supply a fixed integrable majorant for their densities, and the Dunford--Pettis theorem delivers a limit whose pure part vanishes; Theorem \ref{t3} is then obtained in Subsection \ref{subsec:prooft3} by testing that limit against the inward direction of \textbf{(C3)}. Theorem \ref{t4} requires no approximation at all: Subsection \ref{subsec:prooft4} derives a multiplier identity valid in the presence of a charge, for which the action of $\beta$ on absolutely continuous functions must first be identified (Lemma \ref{lem:betacharge}). Subsection \ref{subsec:prooft5}, finally, combines both routes: the initial instant is handled by truncation, whereas the terminal one forces the auxiliary problems to carry charges of their own, so that the limit must be taken in $\Linf^*$.
	
	\subsection{Proof of Theorem \ref{t1}}\label{subsec:prooft1}
	
	\begin{proof}[Proof of Theorem \ref{t1}]
		We begin by following the approximation scheme of \cite{ferreira}. Choose a sequence $\{t_i\}\subset(0,1)$ with $t_i\to0^+$ and, for each $i$, consider the auxiliary problem
		\begin{equation*}
			\text{(P$_i$)}~\left\{\begin{array}{lll}
				&\text{min}\quad  J_i(x,u)=\displaystyle\int_{t_i}^1L(t,x(t),u(t))dt  \\
				& s.t.\quad (x,u)\in W^{1,1}([t_i,1];\R^n)\times L^{\infty}([t_i,1];\R^m),\\
				&\dot x(t)=f(t,x(t),u(t))\quad \text{a.e.},\\
				&g(t,x(t),u(t))\leq 0\ \quad \text{a.e.},\\
				&x(t_i)=\bar x(t_i).
			\end{array}\right.
		\end{equation*}
		Since (P$_i$) shares the constraints of (P), the restriction of $(\bar x,\bar u)$ to $[t_i,1]$ is a weak local minimizer for (P$_i$). Moreover, by hypothesis \textbf{(R)} applied with $\delta=t_i$, this restriction is regular in the sense of Definition \ref{reg}; hence, as recalled after that definition (see \cite{dm2}), the multipliers provided by Theorem \ref{pmp} (in the form of Remark \ref{rem:interval}) can be taken with vanishing pure charge. That is, there exist $\lambda_{0i}\geq 0$, an integrable function $\tilde z_i\geq 0$ and $\tilde\lambda_i\in W^{1,1}([t_i,1];\R^n)$ (so that $\tilde\zeta_i=0$ and $\tilde\beta_i\equiv 0$) satisfying:
		\begin{enumerate}
			\item the nontriviality condition
			\begin{equation}\label{NT}\lambda_{0i}+\|\tilde z_i\|_{L^1}>0;\end{equation}
			\item the complementary slackness condition
			\begin{equation}\label{cs}\tilde z_i(t)\bar g(t)=0\quad \text{a.e.;}\end{equation}
			\item the transversality condition $\tilde\lambda_i(1)=0$ and the costate equation
			\begin{equation}\label{ce}
				-\dot{\tilde{\lambda}}_i=\bar f_x^T\tilde\lambda_i+\lambda_{0i}\bar L_x^T+\bar g_{x}^T\tilde z_i\quad\text{a.e.;}
			\end{equation}
			\item the stationarity conditions
			\begin{align}
				0&=\bar{f}_{u}^T(t)\tilde\lambda_{i}(t)+\lambda_{0i}\bar L_u^T(t)+\bar g_u^T(t)\tilde z_i(t)\quad\text{a.e.,} \label{scu}\\
				0&=\bar g_u^T\,d\tilde\zeta_i, \label{scc}
			\end{align}
			the second condition holding trivially because $\tilde\zeta_i=0$.
		\end{enumerate}
		We extend the multipliers to $[0,1]$. The pure charge is trivially extended by $\zeta_i(A)=0$ for all $A\in\mathfrak{L}$, hence also $\beta_i\equiv 0$; the remaining multipliers are extended by
		\begin{equation}\label{ext}
			(\lambda_i, z_i)(t)=
			\begin{cases}
				(\tilde\lambda_{i}(t_i),\,0), & t\in [0,t_i),\\
				(\tilde\lambda_{i}(t),\, \tilde z_i(t)), & t\in [t_i,1],
			\end{cases}
		\end{equation}
		so that $(\lambda_i,\beta_i,\lambda_{0i},z_i,\zeta_i)$ belongs to the space
		\[\mc{Y}=\BV{n}\times \BV{n}\times\R_+\times L_+^1([0,1];\R)\times L^\infty_+([0,1];\R)^*.\]
		
		Now, we break down the proof in several steps:
		
		\emph{Step 1: Pointwise bound on $z_i$.} Let $v$ be as in \textbf{(C1)}. Taking in \eqref{scu} the inner product with $v(t)\in\R^m$ and using $\langle \bar f_u^T\tilde\lambda_i, v\rangle=\langle\tilde\lambda_i,\bar f_u v\rangle=0$, valid a.e. on the active set $\{\bar g=0\}$ because $\bar f_u v=0$ there, we obtain
		\[\tilde z_i(t)\,\bar g_u(t)v(t)=-\lambda_{0i}\,\bar L_u(t)v(t)\qquad\text{a.e. on }\{\bar g=0\}\cap[t_i,1].\]
		Since $\tilde z_i\geq 0$ and $\bar g_uv\geq 1$ a.e. on $\{\bar g=0\}$, taking absolute values yields, a.e. on $\{\bar g=0\}\cap[t_i,1]$,
		\[\tilde z_i(t)\ \leq\ \tilde z_i(t)\,\bar g_u(t)v(t)\ =\ \lambda_{0i}\,|\bar L_u(t)v(t)|\ \leq\ \lambda_{0i}\,\|\bar L_u\|_{L^\infty}\,|v(t)|.\]
		By the complementary slackness condition \eqref{cs}, $\tilde z_i$ vanishes a.e. off $\{\bar g=0\}$, so the preceding bound extends to a.e. $t\in[t_i,1]$; hence the extension \eqref{ext} satisfies
		\begin{equation}\label{star}
			0\ \leq\ z_i(t)\ \leq\ \lambda_{0i}\,\|\bar L_u\|_{L^\infty}\,|v(t)|\qquad\text{for a.e. }t\in[0,1].
		\end{equation}
		Integrating \eqref{star} over $[0,1]$ and using $v\in\Lone[m]$,
		\begin{equation}\label{zbound}
			\|z_i\|_{L^1}=\|\tilde z_i\|_{L^1}\ \leq\ \lambda_{0i}\,\|\bar L_u\|_{L^\infty}\,\|v\|_{L^1}=:\lambda_{0i}\,M .
		\end{equation}
		
		\emph{Step 2: Normalization and lower bound for $\lambda_{0i}$.} Because $\tilde\zeta_i=0$, the nontriviality condition \eqref{NT} reduces to $\lambda_{0i}+\|z_i\|_{L^1}>0$, so we may divide the tuple by $a_i:=\lambda_{0i}+\|z_i\|_{L^1}>0$ and assume, without relabeling,
		\[\lambda_{0i}+\|z_i\|_{L^1}=1,\qquad\text{hence in particular }\ \lambda_{0i}\leq 1 .\]
		In view of the fact that \eqref{zbound} holds for the normalized multipliers,
		\[1=\lambda_{0i}+\|z_i\|_{L^1}\leq \lambda_{0i}+\lambda_{0i}M=\lambda_{0i}(1+M),\qquad\text{that is,}\qquad \lambda_{0i}\ \geq\ \frac{1}{1+M}.\]
		Moreover, since now $\lambda_{0i}\leq 1$, the bound \eqref{star} gives the fixed integrable majorant
		\begin{equation}\label{domination}
			0\ \leq\ z_i(t)\ \leq\ h(t):=\|\bar L_u\|_{L^\infty}\,|v(t)|\qquad\text{for a.e. }t\in[0,1],\qquad h\in\Lone,
		\end{equation}
		independent of $i$.
		
		\emph{Step 3: Uniform bound on the costate.} With $\tilde\beta_i\equiv 0$, integrating \eqref{ce} from $t$ to $1$ and using $\tilde\lambda_i(1)=0$ gives
		\[\tilde\lambda_i(t)=\int_t^1\bigl(\bar f_x(s)^T\tilde\lambda_i(s)+\lambda_{0i}\bar L_x(s)^T+\bar g_x(s)^T\tilde z_i(s)\bigr)\,ds,\]
		so that, with $a:=\|\bar f_x\|_{L^\infty}$ and the normalization $\lambda_{0i}+\|\tilde z_i\|_{L^1}=1$,
		\[|\tilde\lambda_i(t)|\leq \|\bar L_x\|_{L^\infty}+\|\bar g_x\|_{L^\infty}\|\tilde z_i\|_{L^1}+a\int_t^1|\tilde\lambda_i(s)|\,ds\leq \bigl(\|\bar L_x\|_{L^\infty}+\|\bar g_x\|_{L^\infty}\bigr)+a\int_t^1|\tilde\lambda_i(s)|\,ds.\]
		By Gr\"onwall's inequality,
		\begin{equation}\label{lambdabound}
			\|\tilde\lambda_i\|_{L^\infty}\ \leq\ \bigl(\|\bar L_x\|_{L^\infty}+\|\bar g_x\|_{L^\infty}\bigr)e^{a}=:C_0,
		\end{equation}
		a bound independent of $i$.
		
		\emph{Step 4: Passage to the limit.} We now justify the existence of limits when letting $i\to\infty$. No appeal to Banach limits or subnets is needed here thanks to the regularity of the auxiliary problems and the bounds of the previous steps. We extract a single subsequence, not relabelled, along which the three nontrivial components ($\lambda_{0i}$, $z_i$, $\lambda_i$) converge as described in \emph{(a)}--\emph{(c)} below; recall that $\beta_i\equiv 0$ and $\zeta_i=0$ for every $i$.
		
		\emph{(a) The cost multiplier.} After the normalization of Step 2, $\{\lambda_{0i}\}\subset[0,1]$ is a bounded sequence of real numbers; by Bolzano--Weierstrass a subsequence converges, $\lambda_{0i}\to\lambda_0$, and passing to the limit in the bound of Step 2 gives
		\[\lambda_0\ \geq\ \frac{1}{1+M}\ >\ 0 .\]
		
		\emph{(b) The integrable multiplier, and the vanishing of the charge.} By \eqref{domination}, $0\leq z_i\leq h$ with a fixed $h\in\Lone$ independent of $i$; hence $\{z_i\}$ is uniformly integrable and bounded in $L^1$. By the Dunford--Pettis theorem it is relatively weakly compact in $\Lone$, so along a further subsequence $z_i\rightharpoonup z$ weakly in $\Lone$ for some $z\in\Lone$, with $z\geq 0$ a.e.\ because the nonnegative cone of $L^1$ is convex and norm-closed, hence weakly closed. Since the weak limit of $\{z_i\}$ already lies in $L^1$, no purely finitely additive part can be generated in the limit. Indeed, regarding $z_i$ as elements of $\Linf^*$, weak $L^1$ convergence is precisely weak$^*$ convergence to the absolutely continuous charge $z\,d\leb$, whose pure component in the Yosida--Hewitt decomposition \eqref{YH} is zero. We therefore obtain
		\[\zeta=0,\qquad\text{and hence}\qquad \beta\equiv 0\]
		by the defining formula for $\beta$ in Theorem \ref{pmp}(c); the limiting costate will accordingly be absolutely continuous.
		
		\emph{(c) The costate.} By \eqref{lambdabound} the functions $\lambda_i$ are uniformly bounded, $\|\lambda_i\|_{L^\infty}\leq C_0$; being dominated by a constant they are uniformly integrable, so the Dunford--Pettis theorem again yields a weakly-$\Lone$ convergent subsequence. A costate in Theorem \ref{pmp} must, however, be a function of \emph{bounded variation}, and weak $L^1$ convergence alone does not deliver such a limit; we require, in addition, a uniform bound on the variation. Since $\lambda_i$ is constant on $[0,t_i)$ and, on $[t_i,1]$, solves \eqref{ce} with $\tilde\beta_i\equiv 0$, we have for a.e.\ $t$
		\[|\dot\lambda_i(t)|\ \leq\ \|\bar f_x\|_{L^\infty}\,|\lambda_i(t)|+\lambda_{0i}\,|\bar L_x(t)|+|\bar g_x(t)|\,z_i(t),\]
		so that, using \eqref{lambdabound}, $\lambda_{0i}\leq 1$ and $\|z_i\|_{L^1}\leq 1$,
		\[\operatorname{Var}_{[0,1]}(\lambda_i)=\int_0^1|\dot\lambda_i(t)|\,dt\ \leq\ a\,C_0+\|\bar L_x\|_{L^\infty}+\|\bar g_x\|_{L^\infty}\ =:\ V_0,\]
		a bound independent of $i$. Together with $\|\lambda_i\|_{L^\infty}\leq C_0$ this makes $\{\lambda_i\}$ a bounded sequence in $\BV{n}$. By Helly's selection theorem a further subsequence converges pointwise on $[0,1]$ to some $\lambda\in\BV{n}$, with $\operatorname{Var}_{[0,1]}(\lambda)\leq V_0$ by lower semicontinuity of the variation. Since $|\lambda_i|\leq C_0$, dominated convergence upgrades this to convergence $\lambda_i\to\lambda$ in $\Lone$; in particular the weak-$\Lone$ limit of $\{\lambda_i\}$ is this same $\lambda$, now identified as a function of bounded variation; precisely the regularity required of a costate in Theorem \ref{pmp}.
		
		\emph{Step 5: The limit tuple is a multiplier.} It remains to verify that $(\lambda,\beta,\lambda_0,z,\zeta)=(\lambda,0,\lambda_0,z,0)\in\mc Y$ satisfies conditions (a)--(d) of Theorem \ref{pmp}. 
		
		\emph{(a) Costate equation.} Fix $t\in(0,1]$; for every $i$ with $t_i<t$ the integrated costate equation of Step 3 reads
		\[\lambda_i(t)=\int_t^1\bigl(\bar f_x(s)^T\lambda_i(s)+\lambda_{0i}\bar L_x(s)^T+\bar g_x(s)^Tz_i(s)\bigr)\,ds.\]
		Letting $i\to\infty$ we pass to the limit termwise: $\lambda_i(t)\to\lambda(t)$; $\int_t^1\bar f_x^T\lambda_i\to\int_t^1\bar f_x^T\lambda$ since $\lambda_i\to\lambda$ in $\Lone$ and $\bar f_x\in\Linf$; $\lambda_{0i}\int_t^1\bar L_x^T\to\lambda_0\int_t^1\bar L_x^T$; and $\int_t^1\bar g_x z_i\to\int_t^1\bar g_x z$ because $z_i\rightharpoonup z$ weakly in $\Lone$ and $\bar g_x\,\chi_{[t,1]}\in\Linf$. Hence
		\[\lambda(t)=\int_t^1\bigl(\bar f_x(s)^T\lambda(s)+\lambda_0\bar L_x(s)^T+\bar g_x(s)^Tz(s)\bigr)\,ds,\qquad t\in(0,1].\]
		The value at $t=0$ follows by continuity. Thus $\lambda\in\AC{n}$, it satisfies the transversality condition $\lambda(1)=0$ and, upon differentiation, the costate equation in (c) with $d\beta=0$. 
		
		\emph{(b) Stationarity.} Equation \eqref{scu} holds a.e. on $[t_i,1]$, while on $[0,t_i)$ the extension \eqref{ext} gives $z_i=0$ and $\lambda_i\equiv\tilde\lambda_i(t_i)$; hence, for every $\phi\in\Linf[m]$ and every $i$,
		\[\int_0^1\left(\lambda_{i}^T(t)\bar f_u(t)+\lambda_{0i}\bar L_u(t)+z_{i}(t)\bar g_u(t)\right)\cdot \phi(t)\,dt=\int_0^{t_i}\left(\lambda_{i}^T(t)\bar f_u(t)+\lambda_{0i}\bar L_u(t)\right)\cdot \phi(t)\,dt,\]
		and, by \eqref{lambdabound} and $\lambda_{0i}\leq 1$, the right-hand side is bounded in absolute value by
		\[\bigl(C_0\,\|\bar f_u\|_{L^\infty}+\|\bar L_u\|_{L^\infty}\bigr)\,\|\phi\|_{L^\infty}\,t_i\ \longrightarrow\ 0.\]
		Therefore, fixing arbitrary $\phi\in\Linf[m]$, letting $i\to\infty$ and using that $\lambda_i\to\lambda$ in $\Lone$, $\lambda_{0i}\to\lambda_0$ and $z_i\rightharpoonup z$ weakly in $\Lone$, yields
		\[\int_0^1\left(\lambda^T(t)\bar f_u(t)+\lambda_{0}\bar L_u(t)+z(t)\bar g_u(t)\right)\cdot \phi(t)\,dt=0,\]
		and since $\phi$ is arbitrary, the first stationarity condition in (d) holds for $(\lambda, \lambda_0, z)$ a.e.; the second, $\int_{[0,1]}\bar g_u\,d\zeta=0$, is trivial since $\zeta=0$. 
		
		\emph{(c) Complementary slackness.} Each $z_i$ vanishes off $\{\bar g=0\}$, so $\int_E z_i=0$ for every measurable $E\subseteq\{\bar g\neq 0\}$, hence $\int_E z=0$; thus $z=0$ a.e.\ on $\{\bar g\neq 0\}$, i.e.\ $z(t)\bar g(t)=0$ a.e., while $\bar g=0$ holds $\zeta$-a.e.\ trivially. 
		
		\emph{(d) Nontriviality.} It is a direct consequence of the inequality $\lambda_0\geq 1/(1+M)>0$.
		
		Consequently $(\lambda_0,\lambda,0,z,0)$ is a tuple of multipliers for $(P)$ with $\zeta=0$, $\beta\equiv 0$ and $\lambda_0\geq 1/(1+M)>0$. Since $\lambda_0>0$, dividing the whole tuple by $\lambda_0$ normalizes the cost multiplier to $1$.
	\end{proof}
	
	\subsection{Proof of Theorem \ref{t2}}\label{subsec:prooft2}
	
	\begin{proof}[Proof of Theorem \ref{t2}]
		Take a sequence of times $\{t_i\}\subset(0,1]$ with $t_i\to0^+$ and the associated auxiliary problems $(P_i)$ exactly as in the proof of Theorem \ref{t1}.
		
		First, we claim that, under \textbf{(C2)}, the restriction of $(\bar x,\bar u)$ to $[t_i,1]$ is a regular minimizer for $(P_i)$ in the sense of Definition \ref{reg}. Indeed, set
		\[v_{0i}(t)=\begin{cases}\dfrac{\bar g_u(t)^T}{|\bar g_u(t)|^2}, & t\in A_\varepsilon\cap[t_i,1],\\[1ex] 0, & \text{otherwise.}\end{cases}\]
		Since $\omega$ is non-decreasing, on $A_\varepsilon\cap[t_i,1]$ we have $|\bar g_u(t)|\geq\omega(t)\geq\omega(t_i)>0$, so that $|v_{0i}(t)|\leq 1/\omega(t_i)$; hence $v_{0i}\in L^\infty([t_i,1];\R^m)$ and $\bar g_u(t)v_{0i}(t)=1$ for a.e. $t\in A_\varepsilon\cap[t_i,1]$. Thus $(P_i)$ is regular, and as recalled after Definition \ref{reg} (see \cite{dm2}), the multipliers $(\tilde\lambda_i, \tilde\beta_i, \lambda_{0i}, \tilde z_i, \tilde\zeta_i)$ of Theorem \ref{pmp} can be taken with vanishing pure charge $\tilde\zeta_i=0$, hence $\tilde\beta_i\equiv 0$ and $\tilde\lambda_i\in W^{1,1}([t_i,1];\R^n)$, satisfying conditions \eqref{NT}--\eqref{scc}.
		
		Now, extend each tuple to $(\lambda_i, \beta_i, \lambda_{0i}, z_i, \zeta_i)$ defined on $[0,1]$ as in the proof of Theorem \ref{t1} so that $\zeta_i=0$, $\beta_i\equiv 0$ on $[0, 1]$ and $z_i(t)=0$ and $\lambda_i(t)=\tilde\lambda_i(t_i)$ on $[0,t_i)$. Since $\zeta_i=0$, the nontriviality condition \eqref{NT} reduces to $\lambda_{0i}+\|\tilde z_i\|_{L^1}>0$, and we normalize the multipliers so that
		\[\lambda_{0i}+\|z_i\|_{L^1}=1\qquad\text{for every }i;\]
		in particular $\lambda_{0i}\le 1$. Proceeding exactly as in the proof of Theorem \ref{t1}, \eqref{lambdabound} holds in the present situation as well; that is,
		\[\|\tilde\lambda_i\|_{L^\infty}\leq \big(\|\bar L_x\|_{L^\infty}+\|\bar g_x\|_{L^\infty}\big)e^{a}=:C_0,\]
		where $a=\|\bar f_x\|_{L^\infty}$, a bound independent of $i$. 
		
		On the other hand, by complementary slackness \eqref{cs}, $\tilde z_i$ is supported on $\{\bar g=0\}\subseteq A_\varepsilon$. On this set $\bar g_u\neq 0$ by assumption \textbf{(C2)}, so we may solve the stationarity condition \eqref{scu} for $\tilde z_i$ by taking the inner product with $\bar g_u^T$:
		\[\tilde z_i(t)=-\frac{\big(\tilde\lambda_i(t)^T\bar f_u(t)+\lambda_{0i}\bar L_u(t)\big)\bar g_u(t)^T}{|\bar g_u(t)|^2}.\]
		Thus, on account of the uniform bound of the costate and the inequality $\lambda_{0i}\le 1$,
		\[0\le \tilde z_i(t)\le \frac{|\tilde\lambda_i(t)|\,|\bar f_u(t)|+\lambda_{0i}|\bar L_u(t)|}{|\bar g_u(t)|}\le\frac{K}{|\bar g_u(t)|}\le\frac{K}{\omega(t)},\qquad K:=C_0\|\bar f_u\|_{L^\infty}+\|\bar L_u\|_{L^\infty},\]
		for a.e. $t\in\{\bar g=0\}$. Therefore the extensions satisfy
		\[0\le z_i(t)\le h(t):=\frac{K}{\omega(t)}\,\chi_{\{\bar g=0\}}(t)\qquad\text{for a.e. }t\in(0,1],\]
		and $h\in \Lone$ because $\int_0^1\omega^{-1}<\infty$. Crucially, this bound is uniform in $i$.
		
		The domination of $z_i$ above places us exactly in the situation of \eqref{domination} in the proof of Theorem \ref{t1}, with $h=K\omega^{-1}\chi_{\{\bar g=0\}}$ in place of $\|\bar L_u\|_{L^\infty}|v|$, and the arguments of Steps 4 and 5 of that proof apply verbatim, with a single exception: the lower bound on the cost multiplier, which we address below. Along a subsequence (not relabelled): $\lambda_{0i}\to\lambda_0\in[0,1]$ by Bolzano--Weierstrass; by the Dunford--Pettis theorem, $z_i\rightharpoonup z$ weakly in $\Lone$ with $z\geq 0$ a.e., so that the weak$^*$ limit of $\{z_i\}$ in $\Linf^*$ is the absolutely continuous charge $z\,d\leb$, whose pure component in the Yosida--Hewitt decomposition \eqref{YH} vanishes: $\zeta=0$ and hence $\beta\equiv 0$. By the uniform bounds on $\|\lambda_i\|_{L^\infty}$ and on $\operatorname{Var}_{[0,1]}(\lambda_i)$ together with Helly's selection theorem, $\lambda_i\to\lambda$ strongly in $\Lone[n]$ with $\lambda\in\BV{n}$. The verification that the limit tuple $(\lambda_0,\lambda,0,z,0)$ satisfies the costate equation, the transversality, stationarity and complementary slackness conditions of Theorem \ref{pmp} is identical to Step 5\,\emph{(a)}--\emph{(c)} of the proof of Theorem \ref{t1} and is not repeated.
		
		Finally, the nontriviality and nondegeneracy are direct consequences of the fact that $\zeta=0$: testing $z_i\rightharpoonup z$ against $\mathbf 1\in\Linf$ and using the normalization, we obtain
		\[\|z\|_{L^1}=\int_0^1 z(t)\,dt=\lim_i\int_0^1 z_i(t)\,dt=\lim_i\|z_i\|_{L^1}=1-\lambda_0,\]
		implying that $\|z\|_{L^1}$ and $\lambda_0$ cannot vanish simultaneously and that
		\[\lambda_0+\|z\|_{L^1}+\zeta([0,1])=1>0.\]
		Since degenerate tuples have $\lambda_0=0$ and $z\equiv 0$, the limit multipliers are nondegenerate.
	\end{proof}
	
	\subsection{Proof of Theorem \ref{t3}}\label{subsec:prooft3}
	
	\begin{proof}[Proof of Theorem \ref{t3}]
		Assume \textbf{(C2)} and let $(\lambda_0,\lambda,0,z,0)$ be the tuple constructed in the proof of Theorem \ref{t2}. Besides the conditions of Theorem \ref{pmp} with $\zeta=0$ and $\beta\equiv 0$, it satisfies the mass identity
		\[\lambda_0+\|z\|_{L^1}=1\]
		and its costate $\lambda$ is absolutely continuous with $\lambda(1)=0$.
		
		Let $(y,v)$ be as in \textbf{(C3)} and set $w:=\bar g_xy+\bar g_uv\in\Linf$. Since $\lambda$ and $y$ are absolutely continuous with $\lambda(1)=0$ and $y(0)=0$,
		\[0=\lambda(1)^Ty(1)-\lambda(0)^Ty(0)=\int_0^1\bigl(\dot\lambda(t)^Ty(t)+\lambda(t)^T\dot y(t)\bigr)\,dt.\]
		Omitting the $t$-dependence, the costate equation (with $d\beta=0$) gives $\dot\lambda^Ty=-\lambda^T\bar f_xy-\lambda_0\bar L_xy-z\,\bar g_xy$ a.e.; the variational equation of \textbf{(C3)} gives $\lambda^T\dot y=\lambda^T\bar f_xy+\lambda^T\bar f_uv$ a.e.; and the first stationarity condition of Theorem \ref{pmp}(d) gives $\lambda^T\bar f_uv=-\lambda_0\bar L_uv-z\,\bar g_uv$ a.e. Adding the three equations and integrating over $[0,1]$ yields the identity
		\begin{equation}\label{fundid}
			\lambda_0\int_0^1\bigl(\bar L_x(t)y(t)+\bar L_u(t)v(t)\bigr)\,dt+\int_0^1 z(t)\,w(t)\,dt=0.
		\end{equation}
		Suppose, by contradiction, that $\lambda_0=0$. Then \eqref{fundid} reduces to $\int_0^1 z\,w\,dt=0$. By the complementary slackness condition, $z$ vanishes a.e.\ off the active set $\{\bar g=0\}$, while on it $z\geq 0$ and $w<0$ a.e.\ by \textbf{(C3)}; hence $z\,w\leq 0$ a.e., and the vanishing of its integral forces $z\,w=0$ a.e., hence $z=0$ a.e. on the active set and therefore on all of $[0,1]$. This contradicts the nontriviality $\lambda_0+\|z\|_{L^1}=1$. Consequently $\lambda_0>0$.
	\end{proof}
	
	\subsection{Proof of Theorem \ref{t4}}\label{subsec:prooft4}
	
	We begin with a technical lemma:
	
	\begin{lemma}\label{lem:betacharge}
		Let $\zeta\geq 0$ be a charge in $\ba$ and let $\beta$ be given by the formula in Theorem \ref{pmp}(c), together with $\beta(1)=0$. Then $\beta$ coincides a.e. on $[0,1]$ with the function $t\mapsto-\int_{(t,1]}\bar g_x(s)^T\,d\zeta$ and, for every $y\in\AC{n}$ with $y(0)=0$,
		\begin{equation}\label{betaid}
			\int_0^1\beta(t)^T\dot y(t)\,dt\;=\;-\int_{[0,1]}\bar g_x(t)\,y(t)\,d\zeta .
		\end{equation}
	\end{lemma}
	
	\begin{proof}
		Write $\gamma(E)\coloneqq\int_E\bar g_x(s)^T\,d\zeta\in\R^n$ for $E\in\mathfrak{L}$, and set
		\[
		\tilde\beta(t)\coloneqq-\gamma\bigl((t,1]\bigr),\qquad t\in[0,1].
		\]
		Splitting each component of $\bar g_x$ into its positive and negative parts exhibits each component of $\tilde\beta$ as a difference of two bounded non-decreasing functions; hence $\tilde\beta\in\BV{n}$, it possesses one-sided limits at every point, and its set $D$ of discontinuity points is at most countable. Since $\zeta$ vanishes on Lebesgue null sets, $\zeta(\{s\})=0$ for every $s\in[0,1]$, so the integrals occurring in the formula of Theorem \ref{pmp}(c) are unambiguous and, for $t\in(0,1)$,
		\[
		\beta(t)=-\lim_{k\to\infty}\gamma\bigl((t+\tfrac1k,1]\bigr)=\lim_{k\to\infty}\tilde\beta\bigl(t+\tfrac1k\bigr)=\tilde\beta(t^+),
		\]
		while $\beta(0)=-\gamma([0,1])=-\gamma\bigl((0,1]\bigr)=\tilde\beta(0)$ and $\beta(1)=0=\tilde\beta(1)$. Therefore $\beta$ and $\tilde\beta$ agree on the set $G\coloneqq\{0,1\}\cup\bigl([0,1]\setminus D\bigr)$, whose complement is countable; in particular they agree a.e., which is the first assertion of the lemma. Consequently neither side of \eqref{betaid} changes if $\beta$ is replaced by $\tilde\beta$, so we assume from now on that $\beta=\tilde\beta$.
		
		The integration by parts formula for functions of bounded variation applied to $\beta$ and $y$ reads
		\[
		\beta(1)^Ty(1)-\beta(0)^Ty(0)=\int_{(0,1]}\beta(t)^T\,dy+\int_{(0,1]}y(t-)^T\,d\beta,
		\]
		both Stieltjes integrals being well defined because $y$ is continuous and $\beta$ is of bounded variation. Using the facts that $\beta(1)=y(0)=0$ and that $y$ is absolutely continuous, the above expression reduces to
		\[
		\int_{0}^1\beta(t)^T\dot y(t)\,dt=-\int_{(0,1]}y(t)^T\,d\beta,
		\]
		thus we prove the lemma by showing the validity of the equation
		\begin{equation}\label{eq:ybeta}
			\int_{(0,1]}y(t)^T\,d\beta(t)=\int_0^1\bar g_x(t)y(t)\,d\zeta.
		\end{equation}
		
		Take a partition $0=t_0<t_1<\cdots<t_{N-1}<t_N=1$ of $[0,1]$ whose points all lie in $G$ (which is possible because $[0,1]\setminus G$ is countable) and arbitrary elements $c_k\in(t_{k-1}, t_k)$ for $k=1,\ldots,N$. The finite additivity of $\gamma$ gives
		\[
		\beta(t_k)-\beta(t_{k-1})=\gamma\bigl((t_{k-1},t_k]\bigr)=\int_{t_{k-1}}^{t_k}\bar g_x(t)^T\,d\zeta,\qquad k=1,\ldots,N,
		\]
		so that the Riemann--Stieltjes sums associated with the integral on the left hand-side of \eqref{eq:ybeta} are
		\[
		\sum_{k=1}^{N}y(c_k)^T\left[\beta(t_{k})-\beta(t_{k-1})\right] = \sum_{k=1}^{N}\int_{t_{k-1}}^{t_k}\bar g_x(t)y(c_k)\,d\zeta .
		\] 
		Since $y$ is continuous and $\beta$ is of bounded variation, that integral is the limit of these sums, along \emph{every} choice of partitions and of intermediate points, when the ``norm'' of the partitions, defined by $\max\left\{ t_k-t_{k-1} \right\}$, tends to zero. Select then a sequence of partitions with points in $G$ and elements 
		\[
		t_0^j<t_1^j<\cdots<t_{N_j}^j\qquad\text{and}\qquad c_k^j\in(t_{k-1}^j, t_k^j) \text{ for } k=1,\ldots,N_j 
		\]
		with $\max\left\{ t^j_{k}-t^j_{k-1} \right\}\to 0$ as $j\to\infty$, so that
		\[
		\lim_{j\to\infty}\sum_{k=1}^{N_j}\int_{t_{k-1}^j}^{t_k^j}\bar g_x(t)y(c_k^j)\,d\zeta=\int_{(0,1]}y(t)^T\,d\beta(t).
		\]
		Observe that, by defining $\psi_j(t)$ as the step function $\psi_j(t):=\sum_{k=1}^{N_j}y(c_k^j)\chi_{(t_{k-1}^j, t_k^j]}(t)$, we have that $\psi_j\to y$ in the norm of $L^\infty$ on account of the uniform continuity of $y$ in $[0,1]$ and the norm of the partitions tending to zero. Furthermore,
		\[
		\lim_{j\to\infty}\sum_{k=1}^{N_j}\int_{t_{k-1}^j}^{t_k^j}\bar g_x(t)y(c_k^j)\,d\zeta=\lim_{j\to\infty}\int_0^1\bar g_x(t)\psi_j(t)\,d\zeta=\int_0^1\bar g_x(t)y(t)\,d\zeta,
		\]
		the first equality by the definition of the integral of the charge $\bar g_xd\zeta$ with respect to the simple function $\psi_j$, and the second being precisely the definition of the integral of $y$ with respect to the charge $\bar g_x^T\,d\zeta$, completing the proof of the lemma.
	\end{proof}
	
	Now we are ready to prove Theorem \ref{t4}:
	
	\begin{proof}[Proof of Theorem \ref{t4}]
		Let $(\varepsilon,\delta,y,v)$ be as in \textbf{(C4)}, set $w:=\bar g_xy+\bar g_uv\in\Linf$, and fix any tuple $(\lambda,\beta,\lambda_0,z,\zeta)$ satisfying the conditions of Theorem \ref{pmp}. We break down the proof in several steps:
		
		\emph{Step 1: Absolutely continuous reduction of the costate.} In the costate equation of Theorem \ref{pmp}(c) the multipliers $\lambda$ and $\beta$ enter only through their sum: $-d(\lambda+\beta)=\bigl[\bar f_x^T\lambda+\lambda_0\bar L_x^T+\bar g_x^Tz\bigr]dt$. Integrating from $t$ to $1$ and using the transversality condition $\lambda(1)=\beta(1)=0$, we conclude that $\lambda+\beta$ coincides a.e. on $[0,1]$ with the absolutely continuous function
		\[p(t):=\int_t^1\bigl(\bar f_x(s)^T\lambda(s)+\lambda_0\bar L_x(s)^T+\bar g_x(s)^Tz(s)\bigr)\,ds,\]
		which satisfies $p(1)=0$ and $-\dot p=\bar f_x^T\lambda+\lambda_0\bar L_x^T+\bar g_x^Tz$ a.e.; in particular $\lambda=p-\beta$ a.e.
		
		\emph{Step 2: Multiplier identity \eqref{fundid} in the presence of a charge.} Since $p$ and $y$ are absolutely continuous with $p(1)=0$ and $y(0)=0$,
		\[0=p(1)^Ty(1)-p(0)^Ty(0)=\int_0^1\bigl(\dot p(t)^Ty(t)+p(t)^T\dot y(t)\bigr)\,dt.\]
		Omitting the $t$-dependence, the variational equation in \textbf{(C4)} and $\lambda=p-\beta$ a.e. give
		\[p^T\dot y=(p-\beta)^T\bigl(\bar f_xy+\bar f_uv\bigr)+\beta^T\dot y=\lambda^T\bar f_xy+\lambda^T\bar f_uv+\beta^T\dot y\quad\text{a.e.},\]
		while Step 1 gives $\dot p^Ty=-\lambda^T\bar f_xy-\lambda_0\bar L_xy-z\,\bar g_xy$ a.e. By the first stationarity condition of Theorem \ref{pmp}(d), $\lambda^T\bar f_uv=-\lambda_0\bar L_uv-z\,\bar g_uv$ a.e.; combining the three previous equations yields
		\[\dot p^Ty+p^T\dot y=-\lambda_0\bigl(\bar L_xy+\bar L_uv\bigr)-z\,w+\beta^T\dot y\quad\text{a.e.}\]
		We integrate over $[0,1]$ and evaluate $\int_0^1\beta^T\dot y\,dt=-\int_{[0,1]}\bar g_xy\,d\zeta$ by Lemma \ref{lem:betacharge} to obtain
		\begin{equation}\label{fundidch}
			\lambda_0\int_0^1\bigl(\bar L_x(t)y(t)+\bar L_u(t)v(t)\bigr)\,dt+\int_0^1z(t)\,w(t)\,dt+\int_0^1w(t)\,d\zeta=0.
		\end{equation}
		Notice that in the last integral of \eqref{fundidch}, we are using the second stationarity condition in Theorem \ref{pmp}(d), namely, that $\int_E\bar g_u(t)\,d\zeta=0$ for all $E\in\mathfrak{L}$, to ensure the equality 
		\[
		\int_0^1\bar g_x(t)y(t)\,d\zeta=\int_0^1\left( \bar g_x(t)y(t)+\bar g_u(t)v(t) \right)\,d\zeta=\int_0^1w(t)\,d\zeta.
		\]
		
		\emph{Step 3: Localization on the near-active set.} By complementary slackness, $z\bar g=0$ a.e., so $z$ vanishes a.e. off $\{\bar g=0\}\subseteq A_\varepsilon$; since $w\leq-\delta$ a.e. on $A_\varepsilon$ and $z\geq 0$, it follows that $z\,w\leq-\delta\,z$ a.e., hence
		\[\int_0^1z(t)\,w(t)\,dt\leq-\delta\,\|z\|_{L^1}.\]
		As for the charge, $A_\varepsilon^c=\{\bar g<-\varepsilon\}\subseteq\{|\bar g|>\varepsilon/2\}$, and the condition $\bar g=0$ $\zeta$-a.e. gives $\zeta(A_\varepsilon^c)\leq\zeta(\{|\bar g|>\varepsilon/2\})=0$. Let $N\subseteq A_\varepsilon$ be the Lebesgue null set where $w>-\delta$; since $\zeta$ vanishes on Lebesgue null sets, $\zeta(N)=0$, so that $\zeta(A_\varepsilon\setminus N)=\zeta(A_\varepsilon)=\zeta([0,1])$ and
		\[\int_{[0,1]}w\,d\zeta=\int_{A_\varepsilon\setminus N}w\,d\zeta\leq-\delta\,\zeta(A_\varepsilon\setminus N)=-\delta\,\zeta([0,1]).\]
		
		\emph{Step 4: Normality and the mass bound.} Substituting the two estimates of Step 3 into \eqref{fundidch},
		\[\delta\bigl(\|z\|_{L^1}+\zeta([0,1])\bigr)\leq\lambda_0\int_0^1\bigl(\bar L_xy+\bar L_uv\bigr)\,dt\leq\lambda_0\int_0^1\bigl|\bar L_xy+\bar L_uv\bigr|\,dt,\]
		which is \eqref{massbound}. If $\lambda_0=0$, this forces $z=0$ a.e. and, by the positivity of $\zeta$, $\zeta=0$, so that $\lambda_0+\|z\|_{L^1}+\zeta([0,1])=0$, contradicting the nontriviality condition of Theorem \ref{pmp}(a). Hence $\lambda_0>0$, thereby proving the theorem.
	\end{proof}
	
	\subsection{Proof of Theorem \ref{t5}}\label{subsec:prooft5}
	
	The initial instant is treated by truncation, as in Theorems \ref{t1}--\ref{t2}; the terminal instant, being itself nonregular, forces the auxiliary problems to carry pure charges, so that the limit must be taken in $\Linf^*$ rather than in $\Lone$. We use Banach limits to handle the limits on the charge part and the additional charge that may appear at the final point $t=1$ due to the mass drift of the $z_i$. We will require the following two technical lemmas.
	
	\begin{lemma}\label{lem:blimcommute}
		Let $\blim$ be a Banach limit and let $(\omega_i)\subset\ba$ be positive charges with $M:=\sup_i\omega_i([0,1])<\infty$. Then, $\omega(E):=\blim\omega_i(E)$, $E\in\mathfrak L$, defines a positive charge $\omega\in\ba$ with $\omega([0,1])\le M$ and for every $f\in\Linf$,
		\begin{equation}\label{blimcommute}
			\int_0^1 f\,d\omega=\blim\int_0^1 f\,d\omega_i.
		\end{equation}
	\end{lemma}
	
	\begin{proof}
		Finite additivity follows from the linearity of $\blim$, and if $\leb(E)=0$ then $\omega_i(E)=0$ for every $i$, hence $\omega(E)=0$. For any finite measurable partition $\{F_k\}$ of $[0,1]$, choosing signs $\sigma_k\in\{\pm 1\}$ with $\sigma_k\omega(F_k)=|\omega(F_k)|$ and using linearity and property (iv) of Definition \ref{def:banachlimit},
		\[\sum_k|\omega(F_k)|=\blim_i\sum_k\sigma_k\omega_i(F_k)\leq\sup_i\sum_k|\omega_i(F_k)|\leq M,\]
		so $\|\omega\|_{TV}\leq M$. 
		
		Now, to prove the equation, if $f=\sum_kc_k\chi_{E_k}$ is simple, \eqref{blimcommute} holds with equality term by term, again by linearity of $\blim$. For general $f$, fix $\varepsilon>0$ and take a simple $\psi$ with $|f-\psi|\leq\varepsilon$ a.e. on $[0,1]$, as in the integration convention of Section \ref{fam}. Since all the charges involved vanish on Lebesgue null sets, $\bigl|\int f\,d\omega_i-\int\psi\,d\omega_i\bigr|\leq\varepsilon M$ uniformly in $i$, hence $\bigl|\blim\int f\,d\omega_i-\blim\int\psi\,d\omega_i\bigr|\leq\varepsilon M$ by linearity and property (v) of Definition \ref{def:banachlimit}, and likewise $\bigl|\int f\,d\omega-\int\psi\,d\omega\bigr|\leq\varepsilon M$. Since $\int\psi\,d\omega=\blim\int\psi\,d\omega_i$ exactly, the two sides of \eqref{blimcommute} differ by at most $2\varepsilon M$. In view that $\varepsilon>0$ can be taken arbitrary, equality follows.
	\end{proof}
	
	\begin{lemma}\label{lem:limit}
		Let $\{t_i\}\subset(0,1)$ with $t_i\to0^+$, and for each $i$ let a tuple $(\lambda_i, \beta_i, \lambda_{0i}, z_i, \zeta_i)$ satisfying Theorem \ref{pmp} for the auxiliary problem $(P_i)$ on $[t_i,1]$ (Remark \ref{rem:interval}) be extended to $[0,1]$ by \eqref{ext}, by $\zeta_i(E):=\zeta_i(E\cap[t_i,1])$, and by the formula of Theorem \ref{pmp}(c) for $\beta_i$ and normalize it by $\lambda_{0i}+\|z_i\|_{L^1}+\zeta_i([0,1])=1$. Suppose:
		\begin{enumerate}[label=(\roman*)]
			\item $\|\lambda_i\|_{L^\infty}\le C_0$ and $\operatorname{Var}_{[0,1]}(\lambda_i)\le V_0$, uniformly in $i$;
			\item for every $\delta\in(0,1)$, $(z_i)$ is uniformly integrable on $[0,1-\delta]$;
			\item for every $\delta\in(0,1)$, there is $i_\delta$ with $\zeta_i(E)=0$ for every measurable $E\subseteq[0,1-\delta]$ and every $i\ge i_\delta$.
		\end{enumerate}
		Then the following properties hold: 
		\begin{enumerate}[label=(\alph*)]
			\item There is a subsequence, not relabeled, along which $\lambda_{0i}\to\lambda_0\in[0,1]$, $\lambda_i\to\lambda$ pointwise and strongly in $\Lone$ with $\lambda\in\BV{n}$, and $z_i\rightharpoonup z$ weakly in $L^1([0,1-\delta];\R)$, with $z\ge0$, for every $\delta\in (0,1)$ .
			\item For any Banach limit, define the charge $\omega(E):=\blim\bigl(\int_Ez_i\,d\leb+\zeta_i(E)\bigr)$, where the Banach limit is taken over the subsequence that guarantees the properties in item (a). Then, the Yosida--Hewitt decomposition of $\omega$ is $d\omega=z\,d\leb+d\zeta$ with $z$ from item (a) and $\zeta$ a pure charge that is left-concentrated at $t=1$ or vanishes. Moreover, setting $\beta$ as in Theorem \ref{pmp}(c) using this $\zeta$, the tuple $(\lambda,\beta,\lambda_0,z,\zeta)$ satisfies conditions (a)--(d) of Theorem \ref{pmp} for $(\bar x,\bar u)$ on $[0,1]$ and satisfies the total mass identity $\lambda_0+\|z\|_{L^1}+\zeta([0,1])=1$.
		\end{enumerate}
	\end{lemma}
	
	\begin{proof}
		The proof of (a) follows by classical results. By (i) and Helly's selection theorem a subsequence of $(\lambda_i)$ converges pointwise to some $\lambda\in\BV{n}$ with $\operatorname{Var}_{[0,1]}(\lambda)\le V_0$; dominated convergence upgrades this to $\lambda_i\to\lambda$ in $\Lone$. By Bolzano--Weierstrass we take a further subsequence so that $\lambda_{0i}\to\lambda_0\in[0,1]$. By (ii) and the Dunford--Pettis theorem, a diagonal extraction over $\delta=1/k$ gives $z_i\rightharpoonup z$ weakly in $L^1([0,1-\delta];\R)$ for every $\delta$, with the same $z\in\Lone$, $z\ge0$ a.e. for any $\delta$, and $\int_0^{1-\delta}z\,dt=\lim\int_0^{1-\delta}z_i\,dt\le1$.
		
		To prove (b), we begin by defining $d\omega_i:=z_i\,d\leb+d\zeta_i$. We have $\omega_i([0,1])=\|z_i\|_{L^1}+\zeta_i([0,1])\le1$, so Lemma \ref{lem:blimcommute} applies: $\omega\in\ba$ is positive with $\omega([0,1])\le1$ and, for the Yosida--Hewitt decomposition $d\omega=z\,d\leb+d\zeta$, $z$ and $\zeta$ are nonnegative. For any $\delta\in(0,1)$, measurable $E\subseteq[0,1-\delta]$, and $i\ge i_\delta$, (iii) gives $\zeta_i(E)=0$, so $\omega(E)=\blim\int_Ez_i\,dt=\int_Ez\,dt$, the last equality by weak $L^1$ convergence on $[0,1-\delta]$ tested against $\chi_E\in\Linf$. Thus $\omega$ restricted to $[0,1-\delta]$ is the absolutely continuous charge $z\,d\leb$; its pure part $\zeta$ therefore vanishes on every measurable subset of $[0,1-\delta]$, implying that $\zeta(E)=0$ for every measurable $E\subseteq[0,b]$ with $0<b<1$. Hence, either $\zeta=0$ or $\zeta$ is left-concentrated at $t=1$; i.e. $\zeta\bigl((1-\varepsilon,1)\bigr)=\zeta([0,1])>0$ for every $\varepsilon>0$.
		
		Now, let us prove that items (a)--(d) of Theorem \ref{pmp} hold:
		
		\emph{Unitary mass.} By linearity of $\blim$ on the constant-$1$ normalized masses, $\lambda_0+\omega([0,1])=\blim\bigl(\lambda_{0i}+\|z_i\|_{L^1}+\zeta_i([0,1])\bigr)=\blim 1=1$, and $\omega([0,1])=\|z\|_{L^1}+\zeta([0,1])$; hence $\lambda_0+\|z\|_{L^1}+\zeta([0,1])=1$, which is condition (a).
		
		\emph{Costate equation.} By Step 1 of the proof of Theorem \ref{t4}, applied on $[t_i,1]$, the sum $\lambda_i+\beta_i$ coincides a.e. with the absolutely continuous function
		\[p_i(t)=\int_t^1\bigl(\bar f_x(s)^T\lambda_i(s)+\lambda_{0i}\bar L_x(s)^T+\bar g_x(s)^Tz_i(s)\bigr)\,ds,\qquad p_i(1)=0.\]
		Let $t\in(0,1)$ be a point at which the identification of Lemma \ref{lem:betacharge} holds simultaneously for $\zeta$ and for every $\zeta_i$; almost every $t$ has this property, since the exceptional set is a countable union of countable sets. For every $i$ with $t_i<t$ we may then substitute $\beta_i(t)=-\int_{(t,1]}\bar g_x^T\,d\zeta_i$ into $\lambda_i(t)+\beta_i(t)=p_i(t)$ and, recalling that $\zeta_i$ vanishes on Lebesgue null sets, gather the two contributions of $z_i$ and $\zeta_i$ into the single charge $\omega_i$:
		\begin{equation}\label{costatecharge}
			\lambda_i(t)=\int_t^1\bigl(\bar f_x(s)^T\lambda_i(s)+\lambda_{0i}\bar L_x(s)^T\bigr)\,ds+\int_{(t,1]}\bar g_x(s)^T\,d\omega_i .
		\end{equation}
		We now apply $\blim$ to \eqref{costatecharge}. On the left, $\lambda_i(t)\to\lambda(t)$ by (a), so the Banach limit is the ordinary limit $\lambda(t)$; in the first integral on the right, $\lambda_i\to\lambda$ in $\Lone[n]$ with $\bar f_x\in\Linf$ and $\lambda_{0i}\to\lambda_0$, so that term converges as well; and the last integral obeys $\blim\int_{(t,1]}\bar g_x^T\,d\omega_i=\int_{(t,1]}\bar g_x^T\,d\omega$ by Lemma \ref{lem:blimcommute} applied with $f=\bar g_x^T\chi_{(t,1]}\in\Linf[n]$. Splitting $d\omega=z\,d\leb+d\zeta$ in the resulting identity and invoking Lemma \ref{lem:betacharge} once more, this time for $\zeta$,
		\[\int_{(t,1]}\bar g_x^T\,d\omega=\int_t^1\bar g_x(s)^Tz(s)\,ds+\int_{(t,1]}\bar g_x(s)^T\,d\zeta=\int_t^1\bar g_x(s)^Tz(s)\,ds-\beta(t),\]
		hence
		\[(\lambda+\beta)(t)=p(t):=\int_t^1\bigl(\bar f_x(s)^T\lambda(s)+\lambda_0\bar L_x(s)^T+\bar g_x(s)^Tz(s)\bigr)\,ds\qquad\text{for a.e. }t\in(0,1),\]
		with $p(1)=0$. Thus, $\lambda+\beta$ coincides a.e. with the absolutely continuous function $p$, which satisfies $-\dot p=\bar f_x^T\lambda+\lambda_0\bar L_x^T+\bar g_x^Tz$ a.e.; this is the costate equation of Theorem \ref{pmp}(c). As for the transversality conditions, $\lambda(1)=0$ because $\lambda_i(1)=0$ for every $i$ and $\lambda_i\to\lambda$ pointwise, while $\beta(1)=0$ by the formula defining $\beta$.
		
		\emph{Stationarity.} Both conditions in Theorem \ref{pmp}(d) follow from a single identity for the limit charge $\omega$, namely
		\begin{equation}\label{statcharge}
			\int_E\bar g_u(t)\,d\omega=-\int_E\bigl(\lambda(t)^T\bar f_u(t)+\lambda_0\bar L_u(t)\bigr)\,dt\qquad\text{for every }E\in\mathfrak L .
		\end{equation}
		To prove \eqref{statcharge}, fix $E\in\mathfrak L$ and $i$. The charge form of the second stationarity condition on $[t_i,1]$, together with the extension $\zeta_i(\,\cdot\,)=\zeta_i(\,\cdot\cap[t_i,1])$, gives $\int_E\bar g_u\,d\zeta_i=\int_{E\cap[t_i,1]}\bar g_u\,d\zeta_i=0$, while the first stationarity condition \eqref{scu} on $[t_i,1]$ reads $z_i\bar g_u=-\bigl(\lambda_i^T\bar f_u+\lambda_{0i}\bar L_u\bigr)$ a.e. there. Since moreover $z_i=0$ on $[0,t_i)$,
		\[\int_E\bar g_u\,d\omega_i=\int_E\bar g_u z_i\,dt+\int_E\bar g_u\,d\zeta_i=-\int_{E\cap[t_i,1]}\bigl(\lambda_i^T\bar f_u+\lambda_{0i}\bar L_u\bigr)\,dt=-\int_E\bigl(\lambda_i^T\bar f_u+\lambda_{0i}\bar L_u\bigr)\,dt+r_i,\]
		with $|r_i|\leq\bigl(C_0\|\bar f_u\|_{L^\infty}+\|\bar L_u\|_{L^\infty}\bigr)t_i\to0$ by (i) and $\lambda_{0i}\leq1$; applying $\blim$ yields \eqref{statcharge}: on the left, Lemma \ref{lem:blimcommute} with $f=\bar g_u\chi_E\in\Linf[m]$ gives $\blim\int_E\bar g_u\,d\omega_i=\int_E\bar g_u\,d\omega$ on the right, $\lambda_i\to\lambda$ in $\Lone[n]$, $\bar f_u\in\Linf$ and $\lambda_{0i}\to\lambda_0$ make the sequence convergent, so that its Banach limit is its ordinary limit.
		
		Taking the Yosida--Hewitt decomposition $d\omega=zdm+d\zeta$ as in the proof of the costate equation above, we obtain from \eqref{statcharge}
		\[\int_E\bigl(\lambda(t)^T\bar f_u(t)+\lambda_0\bar L_u(t)+z(t)\bar g_u(t)\bigr)\,dt=-\int_E\bar g_u(t)\,d\zeta\qquad\text{for all }\  E\in\mathfrak{L}.\]
		Thus, we can equate the total variation of the \emph{measure} on the left hand-side with the total variation of the \emph{pure charge} on the right hand-side:
		\[
		\bigl|\lambda(t)^T\bar f_u(t)+\lambda_0\bar L_u(t)+z(t)\bar g_u(t)\bigr|\,dt=|\bar g_u(t)|\,d\zeta.
		\]
		By a straightforward consequence of the definition of a pure charge (see Definition \ref{charges}(iv)), this is only possible if both are the null measure, yielding exactly both stationarity conditions in Theorem~\ref{pmp}(d).
		
		\emph{Complementary slackness.} Fix $\varepsilon>0$ and let $E\subseteq\{|\bar g|>\varepsilon\}$ be measurable. For every $i$, the two slackness conditions on $[t_i,1]$ give $z_i=0$ a.e. on $\{\bar g\neq0\}$ and $\bar g=0$ $\zeta_i$-a.e., the latter meaning precisely $\zeta_i(\{|\bar g|>\varepsilon\})=0$ for all $\varepsilon$ by Definition \ref{charges}(vi) and the positivity of $\zeta_i$; consequently
		\[\omega_i(E)=\int_Ez_i\,dt+\zeta_i(E)=0\qquad\text{for every }i,\qquad\text{hence}\qquad \omega(E)=\blim\omega_i(E)=0 .\]
		Since $d\omega=z\,d\leb+d\zeta$ with $z\geq0$ a.e. and $\zeta\geq0$, both summands must vanish: $\int_Ez\,dt=0$ and $\zeta(E)=0$ for every measurable $E\subseteq\{|\bar g|>\varepsilon\}$. The first, applied with $E=\{|\bar g|>\varepsilon\}$ and $\varepsilon\to0^+$, gives $z=0$ a.e. on $\{\bar g\neq0\}$, that is, $z\,\bar g=0$ a.e.; the second gives $\zeta(\{|\bar g|>\varepsilon\})=0$ for every $\varepsilon>0$, that is, $\bar g=0$ $\zeta$-a.e. This is condition (b), completing the verification.
	\end{proof}
	
	Now, we proceed with the proof of Theorem \ref{t5}:
	
	\begin{proof}[Proof of Theorem \ref{t5}]
		Choose $\{t_i\}\subset(0,1)$ with $t_i\to0^+$ and take the auxiliary problems $(P_i)$ on $[t_i,1]$ exactly as in the proof of Theorem \ref{t1}. The restriction of $(\bar x,\bar u)$ to $[t_i,1]$ is a weak local minimizer for $(P_i)$, so Theorem \ref{pmp} (Remark \ref{rem:interval}) furnishes multipliers $(\lambda_i,\beta_i,\lambda_{0i},z_i,\zeta_i)$; unlike in Theorems \ref{t1}--\ref{t2}, the terminal instant is nonregular and $\zeta_i$ need not vanish.
		
		\emph{Step 1: Concentration of the auxiliary charges.} Fix $\delta\in(0,1)$. By \textbf{(R-mid)}, the restriction of $(\bar x,\bar u)$ to $[t_i,1-\delta]\subset(0,1)$ is regular for every $i$ with $t_i<1-\delta$; the transcription of Lemma \ref{lem:chargesupport} to $[t_i,1]$ then gives $\zeta_i(E)=0$ for every measurable $E\subseteq[t_i,1-\delta]$. After the extension by $\zeta_i(E)=\zeta_i(E\cap[t_i,1])$, hypothesis (iii) of Lemma \ref{lem:limit} holds.
		
		\emph{Step 2: Uniform bounds.} Extend and normalize the multipliers as in
		Lemma \ref{lem:limit}, so that $\lambda_{0i}+\|z_i\|_{L^1}+\zeta_i([0,1])=1$;
		in particular each of the three terms is bounded by $1$. By Step 1 of the proof
		of Theorem \ref{t4}, applied on $[t_i,1]$, we have $\lambda_i=p_i-\beta_i$ there,
		with
		\[p_i(t)=\int_t^1\bigl(\bar f_x^T\lambda_i+\lambda_{0i}\bar L_x^T
		+\bar g_x^Tz_i\bigr)\,ds,\qquad
		\beta_i(t)=-\int_{(t,1]}\bar g_x^T\,d\zeta_i .\]
		Since $|\beta_i(t)|\le\|\bar g_x\|_{L^\infty}\zeta_i([0,1])$, writing
		$a:=\|\bar f_x\|_{L^\infty}$ and
		$c:=\max\{\|\bar L_x\|_{L^\infty},\|\bar g_x\|_{L^\infty}\}$, the normalization
		gives
		\[|\lambda_i(t)|\le\|\bar L_x\|_{L^\infty}\lambda_{0i}
		+\|\bar g_x\|_{L^\infty}\bigl(\|z_i\|_{L^1}+\zeta_i([0,1])\bigr)
		+a\int_t^1|\lambda_i|\,ds\ \le\ c+a\int_t^1|\lambda_i|\,ds,\]
		so Gr\"onwall's inequality yields $\sup_{[t_i,1]}|\lambda_i|\le c\,e^{a}=:C_0$;
		as the extension \eqref{ext} is constant on $[0,t_i)$ with value
		$\lambda_i(t_i)$, the same bound holds on $[0,1]$.
		
		For the bound on the variation, $\lambda_i$ is
		constant on $[0,t_i)$ and continuous at $t_i$, hence
		\[
		\operatorname{Var}_{[0,1]}(\lambda_i)=\operatorname{Var}_{[t_i,1]}(\lambda_i)
		\le\operatorname{Var}_{[t_i,1]}(p_i)+\operatorname{Var}_{[t_i,1]}(\beta_i).
		\] 
		The first term is bounded by $\int_{t_i}^1|\dot p_i|\,dt\le a\,C_0
		+\|\bar L_x\|_{L^\infty}\lambda_{0i}+\|\bar g_x\|_{L^\infty}\|z_i\|_{L^1}$ on account of the absolute continuity of $p$. For
		the second, the finite additivity of $\zeta_i$ gives
		$\beta_i(s_k)-\beta_i(s_{k-1})=\int_{(s_{k-1},s_k]}\bar g_x^T\,d\zeta_i$ for
		every partition $t_i=s_0<\cdots<s_N=1$, so that
		\[\sum_{k=1}^N\bigl|\beta_i(s_k)-\beta_i(s_{k-1})\bigr|
		\le\|\bar g_x\|_{L^\infty}\sum_{k=1}^N\zeta_i\bigl((s_{k-1},s_k]\bigr)
		=\|\bar g_x\|_{L^\infty}\,\zeta_i([0,1]).\]
		Adding the two estimates and using the normalization once more,
		$\operatorname{Var}_{[0,1]}(\lambda_i)\le a\,C_0+c=:V_0$, a bound independent of
		$i$; this is hypothesis (i). For (ii), on $[0,\varepsilon_0]$, as in the derivation of the majorant in of the proof of Theorem \ref{t2}, using \textbf{(C2-loc)} and complementary slackness, yields the fixed majorant $0\le z_i(t)\le K\,\omega(t)^{-1}=:h(t)$ a.e., with $h\in L^1([0,\varepsilon_0];\R)$; while on $[\varepsilon_0,1-\delta]$, we bound $z_i$ by $0\le z_i\le K\|v_\delta\|_{L^\infty}$ where $v_\delta$ is the functions whose existence is guaranteed by \textbf{(R-mid)} for the interval $[\varepsilon_0, 1-\delta]$. Hence $(z_i)$ is uniformly integrable on $[0,1-\delta]$ for every $\delta$, which is hypothesis (ii).
		
		\emph{Step 3: The limit tuple.} By Lemma \ref{lem:limit} there is a tuple $(\lambda,\beta,\lambda_0,z,\zeta)$ satisfying all conditions of Theorem \ref{pmp} on $[0,1]$, with total mass $1$, whose pure charge $\zeta$ is concentrated to $t=1$; that is, $\zeta(E)=0$ for every measurable $E\subseteq[0,b]$, $b<1$.
		
		\emph{Step 4: Normality via \textbf{(C4$\setminus$0)}.} Let $(y,v)$ and $\varepsilon,\delta_\eta$ be as in \textbf{(C4$\setminus$0)}, and set $w:=\bar g_xy+\bar g_uv$. Because the limit tuple satisfies Theorem \ref{pmp}, the identity \eqref{fundidch}, whose derivation in the proof of Theorem \ref{t4} uses only the conditions of Theorem \ref{pmp} and the pair $(y,v)$, applies:
		\[\lambda_0\int_0^1\bigl(\bar L_xy+\bar L_uv\bigr)\,dt+\int_0^1z\,w\,dt+\int_{0}^1w\,d\zeta=0.\]
		Fix $\eta\in(0,1)$. By complementary slackness $z$ vanishes a.e.\ off $\{\bar g=0\}\subseteq A_\varepsilon$, where $w\le0$; hence $\int_0^\eta z\,w\,dt\le0$, and on $[\eta,1]$, where $w\le-\delta_\eta$ a.e.\ on $A_\varepsilon$, $\int_\eta^1z\,w\,dt\le-\delta_\eta\int_{\eta}^1z\,dt$. As in Step 3 of the proof of Theorem \ref{t4}, $\zeta$ is carried by $A_\varepsilon$ off a Lebesgue (hence $\zeta$) null set, and being concentrated at $t=1\in[\eta,1]$ it satisfies $\int_{0}^1w\,d\zeta\le-\delta_\eta\,\zeta([0,1])$. Substituting,
		\begin{align*}
			\delta_\eta\left(\zeta([0,1])+\int_{\eta}^1z\,dt\right)&\le-\int_0^\eta z\,w\,dt-\int_\eta^1z\,w\,dt-\int_{0}^1w\,d\zeta \\
			&=\lambda_0\int_0^1\bigl(\bar L_xy+\bar L_uv\bigr)\,dt \\
			&\le\lambda_0\int_0^1\bigl|\bar L_xy+\bar L_uv\bigr|\,dt,
		\end{align*}
		which is \eqref{massboundmixed}. If $\lambda_0=0$, this forces $\zeta([0,1])=0$ and $\int_{\eta}^1z\,dt=0$ for every $\eta$, hence $z=0$ a.e.\ on $[0,1]$; together with $\zeta=0$ the total mass vanishes, contradicting $\lambda_0+\|z\|_{L^1}+\zeta([0,1])=1$ from Step 3. Therefore $\lambda_0>0$, completing the proof.
	\end{proof}
	
	\section{Conclusion}\label{sec:conclusion}
	
	We have studied the degeneracy phenomenon for optimal control problems with a single mixed inequality constraints that fails to be regular at an isolated instant, working within the nonregular maximum principle of \cite{be2021}, where purely finitely additive set functions appear as multipliers. The analysis proceeded in two stages.
	
	First, through the case study of Subsection \ref{subsec:example} we showed that nonregularity at a single instant can render the necessary conditions completely uninformative: every admissible process of the model problem satisfies them with multipliers whose only nonzero component is a pure charge concentrated at the nonregular instant. We further established that this degeneration is intrinsic rather than an artifact of how the conditions are derived: the normalized multipliers of the natural regular approximations admit no weak$^*$ limit in $\Linf^*$, and every generalized limit, in the sense of Banach limits, is again such a degenerate charge (Proposition \ref{prop:limits}). This brought into light the failure of weak$^*$ \emph{sequential} compactness in $\Linf^*$, which we handled by Banach limits.
	
	Second, guided by the mechanism exposed, we proposed verifiable nondegeneracy conditions. Condition \textbf{(C1)}, requiring an integrable direction $v$ with $\bar f_u v=0$ and $\langle\bar g_u,v\rangle\geq 1$ on the active set, guarantees normal, charge-free multipliers (Theorem \ref{t1}). Condition \textbf{(C2)} relaxes it by dropping the requirement $\bar f_u v=0$ and retaining only an integrable monotonous lower bound near the singular time, $|\bar g_u(t)|\geq\omega(t)$ with $\int_0^1\omega^{-1}<\infty$: it secures nondegenerate, charge-free multipliers (Theorem \ref{t2}), although normality may now be lost. This integrability threshold admits Hölder degenerations $|\bar g_u|\geq c\,t^\alpha$ with $\alpha<1$ but excludes the rate $\bar g_u\sim t$, which, as the case study and the example of Section \ref{sec:examples} shows, lies exactly at the onset of complete degeneration. Finally, the inward-pointing condition \textbf{(C3)} upgrades the conclusion of \textbf{(C2)} to normality, $\lambda_0\neq 0$ (Theorem \ref{t3}), so that a genuine maximum principle, free of the cost-multiplier degeneracy, is recovered.

	The terminal instant behaves differently. At a terminal nonregular instant with $\bar g_x(1)\neq 0$, the charge cannot in general be nullified (Example \ref{ex:be43}), and Theorem \ref{t4} secures normality together with a bound on its mass; the terminal balance identity of Proposition \ref{prop:termbalance} characterizes its mass exactly, and shows that at a non-integrable terminal instant the alternative between forced abnormality and a forced charge is decided by the problem data alone. When both endpoints are nonregular at once, the initial instant of the type governed by Theorems \ref{t1}--\ref{t3} and the terminal instant beyond their reach, Theorem \ref{t5} bridges the two regimes, securing a normal tuple whose charge is confined to the terminal instant, with the mass bound \eqref{massboundmixed}; the scalar problem (G$_{\alpha,\gamma}$) of Example \ref{ex:mixed} shows the hypotheses to be realizable and the conclusion to be sharp, existence rather than universality. Two directions remain open. The first is the analogue for a nonregular set accumulating at, or spread over, the interior, where the pure charge becomes genuinely infinite-dimensional and the Banach limit is necessary. The second is the vector-constraint case, in which the scalar lower bound $|\bar g_u|\ge\omega$ of \textbf{(C2)} must be replaced by a bound on the smallest singular value of the Jacobian of the active constraints, the two coinciding precisely when a single constraint is active.
	
	To the best of our knowledge, the results above provide what appears to be the first set of nondegeneracy and normality conditions for control problems with nonregular mixed constraints. Among the factors that trigger the undesirable behaviors are an integrability threshold on $|\bar g_u|^{-1}$ at the nonregular instant, the position of that instant in the interval and, at the terminal instant, a single scalar identity between the problem data.
	
	\bibliographystyle{plain}
	\bibliography{mybib_1}
\end{document}